\documentclass[a4paper,12pt,oneside,reqno]{amsart}
\usepackage{hyperref}
 \usepackage{amscd,amsfonts,amssymb,amsmath,amsthm,bbm,calc}
 \usepackage{enumerate,epsf,epsfig,enumerate}
\usepackage{graphicx}
\usepackage{version,latexsym}
\usepackage{fancyhdr}
\usepackage{indentfirst}
\usepackage{newlfont}
\usepackage{latexsym}
\usepackage[headinclude,DIV13]{typearea}
\usepackage{psfrag}

 \usepackage[english]{babel}
\areaset{15.1cm}{25.0cm}
 \makeatletter
 \@addtoreset{equation}{section}
 \makeatother

 \makeatletter
 \@addtoreset{enunciato}{section}
 \makeatother

 \newcounter{enunciato}[section]

 \newtheorem{ittheorem}{Theorem}
 \newtheorem{itcorollary}{Corollary}
 \newtheorem{itlemma}{Lemma}
 \newtheorem{itproposition}{Proposition}
 \newtheorem{itdefinition}{Definition}
 \newtheorem{itremark}{Remark}
 \newtheorem{itclaim}{Claim}

 \newenvironment{theorem}{\addtocounter{enunciato}{1}
 \begin{ittheorem}}{\end{ittheorem}}

 \newenvironment{corollary}{\addtocounter{enunciato}{1}
 \begin{itcorollary}}{\end{itcorollary}}

 \newenvironment{lemma}{\addtocounter{enunciato}{1}
 \begin{itlemma}}{\end{itlemma}}

 \newenvironment{proposition}{\addtocounter{enunciato}{1}
 \begin{itproposition}}{\end{itproposition}}

 \newenvironment{definition}{\addtocounter{enunciato}{1}
 \begin{itdefinition}}{\end{itdefinition}}

 \newenvironment{remark}{\addtocounter{enunciato}{1}
 \begin{itremark}}{\end{itremark}}

 \newenvironment{claim}{\addtocounter{enunciato}{1}
 \begin{itclaim}}{\end{itclaim}}

\newcommand{\be}{\begin{equation}}
 \newcommand{\ee}{\end{equation}}

\newcommand{\bea}{\begin{eqnarray}}
\newcommand{\eea}{\end{eqnarray}}

 \newcommand{\bl}[1]{\begin{lemma}\label{#1}}
 \newcommand{\el}{\end{lemma}}

 \newcommand{\br}[1]{\begin{remark}\label{#1}}
 \newcommand{\er}{\end{remark}}

 \newcommand{\bt}[1]{\begin{theorem}\label{#1}}
 \newcommand{\et}{\end{theorem}}

 \newcommand{\bd}[1]{\begin{definition}\label{#1}}
 \newcommand{\ed}{\end{definition}}

 \newcommand{\bcl}[1]{\begin{claim}\label{#1}}
 \newcommand{\ecl}{\end{claim}}

 \newcommand{\bp}[1]{\begin{proposition}\label{#1}}
 \newcommand{\ep}{\end{proposition}}

 \newcommand{\bc}[1]{\begin{corollary}\label{#1}}
 \newcommand{\ec}{\end{corollary}}

 \newcommand{\bpr}{\begin{proof}}
 \newcommand{\epr}{\end{proof}}

 \newcommand{\bi}{\begin{itemize}}
 \newcommand{\ei}{\end{itemize}}

 \newcommand{\ben}{\begin{enumerate}}
 \newcommand{\een}{\end{enumerate}}

 \def\botcaption#1#2{\medskip\centerline{{\scshape #1.}\kern8pt   
 {\rm #2}}\bigskip}

 \def\no{\noindent}
 \def \1{\mathbbm{1}}
 \def \ba {\begin{array}}
 \def \ea {\end{array}}

\def \D {{\mathbb D}}

 \def \Z {{\mathbb Z}}
 \def \R {{\mathbb R}}
 
 \def \N {{\mathbb N}}
 \def \P {{\mathbb P}}
 
 \def \E {{\mathbb E}}
\def \H {{\mathbb H}}

 \def \cC {{\cal C}}
 \def \cD {{\cal D}}
 
 \def \cF {{\cal F}}

 \def \cL {{\cal L}}

\def \cK {{\cal K}}

 \def \D {{\Delta}}

\def \a' {{\'{a}}}
\def \i' {{\'{\i}}}
\def \o' {{\'{o}}}
\def \é {{\'{e}}}

 \def \a {{\alpha}}
 \def \b {{\beta}}
 
 \def \d {{\delta}}
 \def \e {{\epsilon}}

  \def \s {{\sigma}}

 \def \m {{\mu}}
 \def \t {{\tau}}
 \def \th {{\theta}}
  \def \l {{\lambda}}
  \def \o {{\omega}}
 \def \th {{\theta}}

\def \n {{\eta}}

\def\TH(#1){\label{#1}}       \def\thv(#1){\ref{#1}}  
\def\Eq(#1){\label{#1}}       \def\eqv(#1){(\ref{#1})} 

\def\cg{c_{gap}}

\def\var{\hbox{\rm Var}}

\def\pv{\hbox{P}_{\mbox{var}}}
\def\v{\var}
\def\sfrac#1#2{{\textstyle{#1\over #2}}}

\def\sfrac#1#2{{\textstyle{#1\over #2}}}

\begin{document}

\title[Glauber dynamics on nonamenable graphs: boundary conditions and mixing time]
{Glauber dynamics on nonamenable graphs:\\
boundary conditions and mixing time}
 \author[A. Bianchi]{Alessandra Bianchi}
\address{A. Bianchi \\
Weierstrass Institute for Applied Analysis and Stochastics\\
Mohrenstrasse 39\\ 10117 Berlin, Germany}
\email{bianchi@wias-berlin.de}


\date{\today}

 \begin{abstract}
We study the stochastic Ising model on finite graphs with $n$ vertices
and bounded degree and analyze the effect of boundary conditions on the mixing time.
We show that for all low enough temperatures, the spectral gap
of the dynamics with $(+)$-boundary condition on a
 class of nonamenable graphs, is strictly positive  uniformly in  $n$.
This implies that the mixing time grows at most linearly in $n$.
The class of graphs we consider includes hyperbolic graphs with sufficiently high degree,
where the best upper bound on the mixing time of the free boundary dynamics
is  polynomial in $n$,
with exponent growing with the inverse temperature.
In addition, we construct a graph in this class,
for which the mixing time in the free boundary case
is exponentially large in $n$. This provides a first example
where the mixing time jumps from exponential to
linear in $n$ while passing from free to $(+)$-boundary condition.
These results extend  the analysis of Martinelli, Sinclair and Weitz
to a wider class of nonamenable graphs.

 \end{abstract}

 \maketitle

 \section{Introduction}
The goal of this paper is to analyze the effect of boundary conditions
on the Glauber dynamics for the Ising model on  nonamenable graphs.
We will focus on a particular class of graphs which includes,
among others, hyperbolic graphs
with sufficiently high degree. Before discussing the
motivation and the formulation of the results we shall give some
necessary definitions.

Given a  finite graph $G=(V,E)$, we consider spin configurations
$\s=\{\s_x\}_{x\in V}$ which consist of an assignment of $\pm
1$-values to each vertex of $V$. In the Ising model the probability
of finding the system in a configuration $\s\in \{\pm 1\}^V\equiv
\Omega_G$ is given by the Gibbs measure
\be\label{prima}\m_{G}(\s)\,=\,(Z_G)^{-1}\exp{\left(\b \sum_{(xy)\in
E}\s_x\s_y +\b h\sum_{x\in V}\s_x\, \right)}\, ,\ee
 where $Z_G$ is a normalizing constant, and $\b$ and $h$
 are parameters of the model corresponding,
respectively, to the inverse temperature and to the external field.
Boundary conditions can also be taken into account by fixing the spin
values at some specified boundary vertices of $G$. The term free
boundary is used to indicate that no boundary is specified.

The Glauber dynamics for the Ising model on $G$ is a (discrete or
continuous time) Markov chain on the set of spin configurations
 $\Omega_G$, reversible with respect to the Gibbs measure $\m_G$. The
corresponding generator is given by \be (\cL f) (\s)\,=\, \sum_{x\in
V}c_x(\s)[f(\s^{x})-f(\s)]\,,\ee
 where $\s^{x}$ is the configuration obtained from $\s$ by a spin
flip at the vertex $x$, and $c_x(\s)$ is the jump rate from $\s$ to
$\s^x$.

Beyond of being the basis of Markov chain Monte Carlo algorithms,
the Glauber dynamics  provides a plausible model for the evolution
of the underlying physical system toward the equilibrium. In both
contexts, a central question is to determine the \emph{mixing time},
i.e. the number of steps until the dynamics is close to its
stationary measure.

In the past decades a lot of efforts have been devoted to the study
of the dynamics for the classical Ising model, namely when $G=G_n$
is a cube of size $n$ in the finite-dimensional lattice $\Z^d$, and
a remarkable connection between the equilibrium and the dynamical
phenomena has been pointed out. As an example, on finite $n$-vertex
cubes with free boundary in $\Z^d$, when $h=0$ and $\b$ is smaller
than the critical value $\b_c$ (one-phase region), the mixing time
is of order $\log n$, while for $\b> \b_c$ (phase coexistence
region) it is $\exp(n^{(d-1)/d})$ (\cite{SZ, MO1, MO2, Mar}).

More recently, an increasing attention has been devoted to the study
of spin systems on graphs other than regular lattices. Among the
various motivations which are beyond this new surge of interest, we
stress that many new phenomena only appear when one considers graphs
different from the Euclidean lattices,  thus revealing the presence
of an interplay between the geometry of the graph and the behavior
of statistical systems.

Here we are interested in the problem of the \emph{influence of
boundary conditions} on the mixing time. It has been conjectured
that in the presence of (+)-boundary condition on regular boxes of
the lattice $\Z^d$,  the mixing time should remain at most
polynomial in $n$ for all temperatures rather than
$\exp(n^{(d-1)/d})$ \cite{FH}. But even if some results supporting
this conjecture have been achieved \cite{BM}, a formal proof for the
dynamics on the lattice is still missing.

 However a different scenario can appear if one
replaces the classical lattice structure with different graphs. The
first rigorous result along this direction, has been  obtained
recently by Martinelli, Sinclair and Weitz \cite{MSW} when studying
the Glauber dynamics for the Ising model on regular trees.
With this graph setting and in presence of $(+)$-boundary condition,
they proved in fact that the mixing time remains of order $\log n$
also at low temperatures (phase coexistence region), in contrast to
the free boundary case where it grows polynomially in $n$
\cite{KMP, BKMP}.

In this paper we extend the above result to a class of
nonamenable graphs which includes trees, but also
hyperbolic graphs with sufficiently high degree,
and some suitable constructed expanders.
Specifically, we consider the dynamics on an $n$-vertex ball
of the  graph with $(+)$-boundary condition, and prove
that the spectral gap is $\Omega(1)$ (i.e.  bounded away from
zero uniformly in $n$) for all low enough temperatures and zero
external field. This implies, by classical argument (see, e.g.,
\cite{Sa}), an upper bound of order $n$ on the mixing time.
Notice that this result is in contrast with the behavior
of the free boundary  dynamics on  hyperbolic graphs,
for which the spectral gap is decreasing in $n$ for all
low temperatures, and bounded below by $n^{-\a(\b)}$, with exponent
$\a(\b)$ arbitrarily increasing with $\b$ \cite{KMP, BKMP}.
Moreover, we give an example of an expander, in the above
class of graphs, for which
we prove that the mixing time of the free boundary  dynamics
is at least exponentially large in $n$.
This provides a first rigorous example
of graph where the mixing time shrinks from exponential
to linear in $n$ while passing from free to $(+)$-boundary condition.

We remark that what we believe to be determinant for the result obtained
in \cite{MSW} for the dynamics on trees, is in fact the nonamenability of the graph.
On the other hand, the possible presence of cycles, which are absent on trees,
makes the structure of some nonamenable graphs more similar to classical lattices.
Our results show that cycles are not an obstacle for proving the influence
of boundary conditions on the mixing time.

The work is organized as follows. In section 2 we give some basic
definitions and state the main results. In section 3 we analyze
the system at equilibrium and prove a mixing property of the
plus phase. Then, in section 4, we deduce from this property a lower bound
for the spectral gap of the dynamics and conclude the proof of our main
result. Finally, in section 5, we give an example of a graph satisfying the
hypothesis of the main theorem, and prove for it an exponential lower bound
on  the spectral gap for the free boundary dynamics.

\section{The model: definitions and main result}
\subsection{Graph setting}\label{subsec:hgnotation}
Before describing the class of graphs in which we are interested,
let us fix some notation and recall a few definitions concerning the graph structure.

 Let $G=(V,E)$ be a general (finite or infinite) graph, where $V$ denotes the vertex set and
$E$ the edge set. We will always implicitly assume that $G$ is connected.
The \emph{graph distance} between two vertices
$x,y \in V$ is  defined as the length of the shortest path from $x$
to $y$ and it is denoted by $d\,(x,y)$. If $x$ and $y$ are at
distance one, i.e. if they are
neighbors, we write $x\sim y$.
The set of neighbors of $x$ is denoted by $N_x$, and $|N_x|$
is called the \emph{degree} of $x$. \\
 For a given subset $S\subset V$, let $E(S)$ be the set of all
edges in $E$ which have both their end vertices in $S$ and
define the \emph{induced subgraph} on $S$ by $G(S):=(S, E(S))$. When it
creates no confusion, we will identify $G(S)$ with its vertex set
$S$.

 For $S\subset V$ let us introduce the \emph{ vertex boundary}
of $S$
$$ \partial_V S\,=\,\{x\in V\setminus
S \,:\, \exists y\in S\,\,\mbox{s.t.}\, x\sim y \}\,$$
 and the
\emph{edge boundary} of $S$
$$\partial_E S\,=\,\{e=(x,y)\in E
\,\,\mbox{s.t.}\,\, x\in S\,,\, y\in V\setminus S \}\,.$$
If $G=(V,E)$ is an infinite, locally finite graph, we can
define the \emph{edge isoperimetric constant} of $G$ (also called
\emph{Cheeger constant}) by
 \be\label{isoperimetrica2}
 i_e(G)\,:=\,\inf \left\{\frac{|\partial_E
(S)|}{|S|}\,; \, S\subset
V\mbox{ finite }\right\}\,.
\ee
\begin{definition}\label{amenabilita}
 An infinite graph $G=(V,E)$ is amenable if
 its edge isoperimetric constant  is zero, i.e.
if for every $\e>0$ there is a finite set of vertices $S$ such that
$|\partial_E S|<\e|S|$. Otherwise $G$ is nonamenable.
\end{definition}
Roughly speaking, a nonamenable graph is such that
the boundary of every subgraph is of comparable size to its volume.
 A typical example of amenable graph is the lattice $\Z^d$, while
 one can easily show that  regular trees,
 with branching number bigger than two, are nonamenable.
We emphasize that nonamenability seems to be strongly related to the qualitative
 behavior of models in statistical mechanics.
 See, e.g., \cite{ JS, Ly1, Ly2, Sch} for results concerning the
Ising and the Potts models, and \cite{BS, BS2, HJL, Jo} for
percolation and random cluster models.
\vspace{5mm}

In this paper we focus on a class of nonamenable graphs,
that we call \emph{growing graphs}, defined as follows.
Given an infinite graph $G=(V,E)$ and a vertex $o\in V$,
let  $B_r(o)$ denote  the ball centered in $o$
and with radius $r\in\N$ with respect to the graph distance,
namely the finite subgraph induced on $\{x\in V : d\,(o,x)\leq r\}$,
and let $L_r(o):=\{x\in V:\,d(x,o)=r\}= \partial_V B_{r-1}(o)$.
\begin{definition}\label{def:growing}
An infinite graph $G=(V,E)$ is growing with parameter $g>0$ and root $o\in V$, if
\be\label{growing}
\min_{ x\in L_{r}(o),\, r\in\N
}\left\{
\left|N_x\cap L_{r+1}(o)\right|-\left|N_x \cap B_r(o)\right|
\right\}= g\,.
\ee
We call $G$ a  $(g,o)$-growing graph.
\end{definition}
It is easy to prove that a growing graph in the sense of Definition
\ref{def:growing} is also nonamenable. The simplest example
of growing graph with parameter $g$, is an infinite tree
with minimal vertex degree equal to $g+2$, where the growing property
is satisfied for every choice of the root on the vertex set.
On the other hand, there are many examples of growing graphs
which are not cycle-free. Between them we mention
 hyperbolic graphs,  that we will prove to be growing
provided that the vertex-degree is sufficiently high.

Hyperbolic graphs are a family
of infinite planar graphs characterized by a cycle periodic
 structure. They can be briefly described as follows
 (for their detailed construction see, e.g., \cite{Mag},
 or Section 2 of ref. \cite{RNO}). Consider a planar  graph
 in which each vertex  has the same degree, denoted by $v$, and each face (or
\emph{tile}) is equilateral with constant number of sides denoted by
$s$.  If the parameters $v$ and $s$ satisfy the relation
$(v-2)(s-2)>4$, then the graph can be embedded in the hyperbolic
plane $\H^2$ and  it is called \emph{hyperbolic graph}
with parameters $v$ and $s$. It will be
denoted by $\H(v,s)$.  The typical representation of hyperbolic
graphs  make use of the Poincar\'{e} disc that is in
bi-univocal correspondence with $\H^2$ (see Fig. \ref{fig:esgrafo}).

\begin{figure}[htbp]
\begin{center}
\includegraphics[width=6cm]{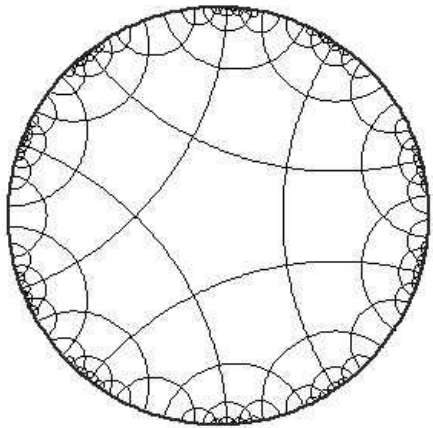}
\end{center}
\caption{{\small The hyperbolic graph $\H(4,5)$ in the Poincar\'{e}
disc representation.}}\label{fig:esgrafo}
\end{figure}

Hyperbolic graphs are nonamenable, with
edge isoperimetric constant  explicitly
computed in \cite{HJL} as a function of $v$ and $s$.
Moreover,  the following  holds:
\begin{lemma}\label{lemma:hyper}
For all couples $(v,s)$ such that $s\geq 4$ and $v\geq 5$,
or $s=3$ and $v\geq 9$,  $\H(v,s)$ is  a $(g,o)$-growing graph
for every vertex $o\in V$ and with parameter $g=g(v,s)$.
\end{lemma}
The proof of this Lemma is postponed to Section 5, where we will
also construct a growing graph that will serve us as
further example of influence
of boundary conditions on the mixing time.

Let us stress, that due to the possible presence of cycles in  a growing graph,
a careful analysis of the correlations between spins will be required.
This is actually the main distinction between our proof and the similar work on trees
\cite{MSW}.

\subsection{Ising model on nonamenable graphs}\label{sec:ising}
The Ising model on   nonamenable graphs
has been investigated in many papers (see, e.g, \cite{Ly2} for a survey).
A  general result, concerning the uniqueness/non-uniqueness
phase transition of the model, is the following \cite{JS}:
\begin{theorem}[Jonasson and Steif]\label{th:JS}
If $G$ is a connected nonamenable graph with bounded degree,
then there exists an inverse temperature $\b_0>0$, depending on the graph,
such that for all $\b\geq\b_0$ there exists an interval
of $h$ where $G$ exhibits a phase transition.
\end{theorem}
Thus, contrary to  what happens on the Euclidian lattice,
the Ising model on nonamenable graphs undergoes a phase transition
also at non zero value of the external field.

Though some   properties of the Ising model are
common to all nonamenable graphs,
the particular behavior of the system may differ from one family
to another one, also depending on other geometric parameters.
Since we will be especially interested in hyperbolic graphs,
we recall briefly the main results
concerning the Ising model on these graphs,
and stress which are the main differences from the model
on classical lattices.

It has been proved (see \cite{SS, Wu, Wu2})
that the Ising model on  $\H(v,s)$ exhibits two different phase
transitions appearing at inverse temperatures $\b_c\leq\b_c'$.
The first one, $\b_c$, corresponds  to the occurrence of a
uniqueness/non-uniqueness phase transition, while
 the second critical temperature refers to a change
 in the properties of the free boundary condition measure $\m^f$.
  Specifically, it is
 defined as
\be\label{secondcp}\b_c'\,:=\,\inf\{\b\geq\b_c\,:\,\m^f=(\m^+ + \m^-)/2\}\,,\ee
where $\m^+$ and $\m^-$ denote the extremal measures obtained by
imposing, respectively, $(+)$- and $(-)$-boundary condition.
As is well explained in \cite{Wu2} (see also \cite{CCST} for more details),
using the Fortuin-Kasteleyn representation it is possible
to show that $\b'_c<\infty$ for all hyperbolic graphs, in contrast
to the behavior of the model on regular trees where $\m^f\ne(\m^+ + \m^-)/2$
for all  finite $\b\geq\b_c$.
From Definition \ref{secondcp} it turns out that for $\b_c\leq\b<\b'_c$,
when this interval is not empty (see \cite{Wu2}),
the measure $\m^f$ is not a
convex combination of $\m^+$ and $\m^-$.
This implies the existence of a translation invariant
Gibbs state different from $\m^+$ and $\m^-$, in contrast
 to what happens on $\Z^d$ \cite{Bod}.

Another interesting result concerning the Ising model on hyperbolic graphs,
is due to Sinai and Series \cite{SS}.
For low enough temperatures and $h=0$, they
proved the existence of uncountably many mutually singular Gibbs states which they
conjectured to be extremal. This points out, once more,
the difference between the system on  hyperbolic graphs
and on  classical lattices, where it is known that the extremal
measures are at most a countable number.\vspace{3mm}

In this paper we are interested in the region of the phase diagram
where the dynamics is highly sensitive to the boundary condition,
namely when the temperature is low and the magnetic field is zero
(phase coexistence region).
Let us explain the model in detail and give the necessary
definitions and notation.

Let $G=(V,E)$ be an infinite $(g,o)$-growing graph
with maximal degree $\Delta$. For any  $r\in\N$,
we denote by $B_r=(V_r , E_r)\subset G$ the ball with radius $r$ centered in $o$.
When it does not create confusion, we identify the subgraphs  of $G$
with their vertex sets.
Given a finite ball $B\equiv B_m$ and an Ising spin configuration
$\t\in\Omega_G$, let $\Omega_B^\t\subset
\{\pm 1\}^{B\cup\partial_V B}$ be the set of configurations that
agree with $\t$ on $\partial_V B$. Analogously,
 for any subset $A\subseteq V_m$ and any $\n\in\Omega_B^\t$, we
 denote by $\Omega_A^\n\subset \{\pm
1\}^{A\cup\partial_V A}$  the set of configurations  that agree with
$\n$ on $\partial_V A$.  The Ising model on $A$ with $\n$-boundary
condition (b.c.) and zero external field is thus specified by the
Gibbs probability measure $\m_A^\n$, with support on $\Omega_A^\n$,
defined as
\be\label{def:gibbsmeasure}\m_A^\n(\s)\,=\,\displaystyle\frac{1}{Z(\b)}\exp(\,\b
\sum_{(x,y)\in E(\overline{A})} \s_x \s_y\,)\,,\ee where $Z(\b)$ is
a normalizing constant and the sum runs over all pairs of
nearest neighbors in the induced subgraph on
$\overline{A}=A\cup\partial_V A$.

Similarly, the Ising model on $A$ with free boundary condition is
specified by the Gibbs measure
 $\m_A$ supported on the set of configurations $\Omega_A:= \{\pm 1\}^{A}$.
 This is defined as in (\ref{def:gibbsmeasure}) by replacing the sum over
$E(\overline{A})$ in a sum over $E(A)$,   namely cutting away the
influence of the boundary $\partial_V A$. Notice that when
$A=V_m$, $\m^\n_{V_m}$ is simply the Gibbs measure on $B$ with
boundary condition $\t$ ($\n$ agrees with $\t$ on
$\partial_VV_m\equiv\partial_V B$) and $\m_{V_m}$ is the Gibbs measure
on $B$ with free boundary condition.

We denote by $\mathcal F_{A}$ the $\s$-algebra  generated by the set
of projections $\{\pi_x\}_{x\in A}$ from $\{\pm 1\}^A$ to $\{\pm
1\}$, where $\pi_x\,:\,\s\mapsto\s_x$, and  write
$f\in\mathcal{F}_A$ to indicate that $f$ is
$\mathcal{F}_A$-measurable. Finally, we recall that if
$f\,:\,\Omega_B^\t\,\rightarrow\,\mathbb R$ is a measurable
function, the \emph{expectation of} $f$ w.r.t. $\m_A^\n$ is given by
$\m_A^\n (f)=\sum_{\s\in\Omega}\m_A^\n(\s)f(\s)\,$
 and the \emph{variance of} $f$ w.r.t. $\m_A^\n$ is
 given by $\var_{A}^\n\,=\,\m_A^\n (f^2)- \m_A^\n (f)^2$.
 We usually think of them as functions of $\n$, that is $\m_A(f)(\n)=
\m_A^\n (f)$ and $\var_A(f)(\n)=\var_A^\n (f)$. In particular
$\m_A(f)\,,\,\var_A(f) \in\mathcal{F}_{A^c}$.

 In the following discussion we will
be concerned with the Ising model on $B$ with $(+)$-b.c. and we will
use the abbreviations $\Omega^+$, $\mathcal F$ and  $\m$ instead of
$\Omega_B^+$, $\mathcal F_{\overline B}$ and $\m_B^+$, and
thus $\m(f)$ and $\var(f)$ instead of $\m_B^+ (f) $ and
$\var_{B}^+(f)$.

\subsection{Glauber dynamics and mixing time}\label{sec:glauber}

The  Glauber dynamics on $B$ with $(+)$-boundary condition is a
continuous time Markov chain $(\s(t))_{t\geq0}$ on $\Omega^+$ with
Markov generator $\cL $  given by
\be\label{generatoreising}(\mathcal{L}f)(\s)\,=\,\sum_{x\in B}
c_x(\s)\left[f (\s^x)-f(\s)\right]\,,\ee where $\s^x$ denotes the
configuration obtained from $\s$ by flipping the spin at the site
$x$ and $c_x(\s)$ is the jump rate from $\s$ to $\s^x$. We sometimes
prefer the short notation $\nabla_x f (\s)=[f(\s^{x})-f(\s)]$.
 The  jump rates are required to be of finite-range,
uniformly positive, bounded, and they should satisfy  the detailed
balance condition  w.r.t. the Gibbs measure $\m$. Although all our
results apply to any choice of jump rates satisfying these
hypothesis, for simplicity we will work with a specific choice
called \emph{heat-bath dynamics}:
 \be\label{rates}
c_x(\s)\,:=\,\m_{x}^\s(\s^x)=\frac{1}{1+\omega_x (\s)}\;\mbox{ where
}\quad \omega_x(\s)\,:=\,\exp(2\b\s_x\,\sum_{y\sim x}\s_y\,).
\ee
 It is easy to check that the Glauber dynamics is ergodic and
 reversible w.r.t. the Gibbs measure $\m$, and so converges to $\m$ by the
 Perron-Frobenius Theorem. The key point is now to determine the rate of
 convergence of the dynamics.

 A useful tool to approach this problem is the \emph{spectral gap}
 of the generator $\cL$, that can be defined as the inverse of
 the first nonzero eigenvalue of $\cL$.
 \begin{remark}
 Notice that the generator $\cL$ is a non-positive
 self-adjoint operator on $\ell^2(\Omega^+,\mu)$. Its spectrum
 thus consists of discrete eigenvalues of finite multiplicity that
 can be arranged as $0=\l_0\geq -\l_1\geq-\l_2\geq\ldots,\geq-\l_{N-1}$, if
 $|\Omega^+|=N$, with $\l_i\geq 0$.
\end{remark}
An equivalent definition of spectral gap is given through the so
called \emph{Poincar\'{e} inequality} for the measure $\m$. For a
function $f:\Omega^+\mapsto \R$, define
 the \emph{Dirichlet form} of $f$ associated to $\cL$ by
 \be\label{dirichlet} \mathcal{D}(f)\,:=\,\frac{1}{2}\sum_{x\in B}\m\left(
 c_x[\nabla_x f]^2\right)\,=\,
\sum_{x\in B} \m(\var_{x}(f))\,,\ee where the second equality holds
under our specific choice of jump rates. The \emph{spectral gap} of the generator,
$\cg(\m)$, is then defined as the inverse of the best constant $c$ in
the \emph{Poincar\'{e} inequality} \be\label{def:poincare}
\v(f)\,\leq\,c\cD(f)\,,\quad\forall f\in \ell^2(\Omega^+,\,\m),\ee
or equivalently
\be\label{cgap}
\cg(\m)\,:=\,\inf\left\{\frac{\mathcal{D}(f)}{\v(f)};
\v(f)\ne 0\right\}\,.
\ee
Denoting by $P_t$ the Markov semigroup associated to $\cL$, with
transition kernel  $P_t(\s,\n)= e^{t\cL}(\s,\n)$, it easy to show
that
 \be\label{cgapdecay}
 \v (P_t f)\,\leq\,e^{-2\cg(\m)
t}\,\v(f)\,.
\ee
The last inequality shows that the spectral gap gives a measure of
the exponential decay of the variance, and justifies the name
\emph{relaxation time} for the inverse of the spectral gap.

Moreover, let $h_t^\s$ denote the density of the distribution at
time $t$ of the process starting at $\s$ w.r.t. $\m$, i.e.
$h_t^\s\,(\n)\,=\,\frac{P_t(\s,\,\n)}{\m(\n)}$.
For $1\leq p\leq
\infty$ and a function $f\in \ell^p(\Omega^+,\,\m)$, let $\|f\|_p$
denote the $\ell^p$ norm of $f$ and define the time of convergence
\be\label{tempomixing}
\t_p\,=\min
\left\{t>0\,:\,\sup_{\s}\|h_t^\s-1\|_p\leq e^{-1}\right\}\,,
\ee
that
for $p=1$ is  called \emph{mixing time}. A well known and useful
result relating $\t_p$ to the spectral gap (see, e.g., \cite{Sa}),
when specializing to the Glauber dynamics yields the following:
\begin{theorem}\label{teo:gapmixingtime}
On an $n$-vertex ball $B\subset G$ with
$(\t)$-boundary condition,
\be
\cg(\m)^{-1}\leq \t_1 \leq \cg(\m)^{-1}\times c n\,,
\ee
where
$\m=\m_B^\t$ and $c$ is a positive constant independent of
$n$.{\hfill$\square$\vskip.5cm}
\end{theorem}
We stress that a different choice of jump rates (here we considered
the heat-bath dynamics) only affects the spectral gap by at most a
constant factor. The bound stated in Theorem \ref{teo:gapmixingtime}
is thus  equivalent, apart for a multiplicative constant, for any
choice of the Glauber dynamics.

Before presenting our main result, we  recall that the Glauber dynamics
for the Ising model on regular trees and hyperbolic graphs has been recently
investigated by Peres et al. \cite{KMP, BKMP}.
In particular, they consider the free boundary dynamics on a finite ball
$B\subset G$, $G$ hyperbolic graph or regular tree,
and prove  that at all
temperatures,  the inverse spectral gap (relaxation time)
scales at most polynomially in the size of $B$, with exponent
$\a(\b) \uparrow\infty$ as $\b\rightarrow\infty$. Let us
stress again that under the same conditions, the dynamics on a cube of
size $n$ in the $d$-dimensional lattice,  relaxes in a time
exponentially large in the surface area $n^{(d-1)/d}$.

\subsection{Main results}
We  are finally in position to state our main results.
\begin{theorem}\label{teo}
Let $G$  be an infinite $(o,g)$-growing graph
with maximal degree $\Delta$. Then, for all $\b\gg 1$,
the Glauber dynamics on the
$n$-vertex ball $B$ with (+)-boundary condition and
zero external field has spectral gap $\Omega (1)$.
\end{theorem}

As a corollary we  obtain that, under the same hypothesis of the
theorem above,  the mixing time of the dynamics is bounded linearly
in $n$ (see  Theorem \ref{teo:gapmixingtime}).

This result, applied to hyperbolic graphs with sufficiently high degree,
 provides a convincing example of the influence of  the boundary condition
 on the mixing time.
 Indeed, for all $\b\gg 1$, due to the fact that the free boundary measure
on $\H(v,s)$ is a convex combination of
$\m^+$ and $\m^-$ (see section \ref{sec:ising}),
 it is not hard to prove that  the spectral gap for
an $n$-vertex ball in the hyperbolic graph
with free boundary condition, is decreasing with $n$.
Most likely it will be of order $n^{-\a(\b)}$, with
$\a(\b)\uparrow\infty$ as $\b\rightarrow\infty$, as
in the lower bound given in \cite{KMP, BKMP}.
The presence of $(+)$-boundary condition thus gives rise to
a jump of the spectral gap, and consequently it speeds up the dynamics.

\noindent \textbf{Remarks.}
\begin{enumerate}
\item[(i)]
\emph{ We recall that on $\Z^d$ not much is known about the
mixing time  when $\b>\b_c$, $h=0$ and the boundary condition is $(+)$,
though it has been conjectured that the it should be
polynomial in $n$ (see \cite{FH} and \cite{BM}).}
\item[(ii)]
\emph{A result similar to Theorem \ref{teo}  has been obtained for
the spectral gap, and thus for the mixing time, of the dynamics on a
regular $b$-ary tree (see \cite{MSW}). In particular it has been
proved that while under free-boundary condition the mixing time on a
tree of size $n$ jumps from $\log n$ to $n^{\Theta(\b)}$ when
passing a certain critical temperature, it remains of order $\log n$
at all temperatures and at all values of the magnetic field under
$(+)$-boundary condition. However we stress that while trees do not
have any cycle, growing graphs, and in general nonamenable graphs,
can have many cycles, as well as the Euclidean lattices.
The theorem above can thus be looked upon
as an extension of this result to a class of graphs which in some
respects are similar to Euclidean lattices. }
\item[(iii)]
\emph{At high enough temperatures (one phase region) the spectral gap
of the dynamics on a ball $B\subset G$, where
$G$ is an infinite graph with bounded degree,
is $\Omega(1)$ for all boundary conditions,
as can be proved by  path coupling techniques  \cite{Wei}.
This suggests that the result  of Theorem \ref{teo} should hold
for all temperatures, as for the  dynamics  on regular trees, and not only
for $\b\gg 1$.
At the moment, what happens in the intermediate region of temperature,
still remains an open question.}
\end{enumerate}

 The following result provides a further example
 of influence of boundary conditions on the dynamics.
\begin{theorem}\label{th:esempio}
For all finite $g\in\N$, there exists an infinite $(g,o)$-growing graph
$G$ with bounded degree, such that, for all $\b\gg 1$, the Glauber
dynamics on the $n$-vertex ball $B$ with free boundary condition
and zero external field has spectral gap $O\left(e^{-\th n}\right)$, with
$\th=\th(g,\b)>0$.
\end{theorem}
Combining this with Theorem \ref{teo:gapmixingtime}
and Theorem  \ref{teo}, we get a first rigorous example where the
mixing time jumps abruptly from exponential to linear in $n$ while
 passing from  one boundary condition to another.

We now proceed to sketch briefly the ideas and techniques used along
the paper.\\
 The proof of our main result, Theorem \ref{teo}, is based on the
variational definition of the spectral gap and it is aimed to show
that the Gibbs measure relative to the system satisfies a
Poincar\'{e} inequality with constant $c$ independent of the size of
$B$. We will first analyze the equilibrium properties of the system
conditioned on having $(+)$-boundary, and under this condition we
will deduce a special kind of correlation decay between spins.
The proof of this spatial mixing property rests on a
disagreement argument and on a Peierls type
argument.

The second main step to prove Theorem \ref{teo}, is deriving a
Poincar\'{e} inequality for the Gibbs measure
from the obtained notion of spatial mixing. This will be achieved by
first deducing, via coupling techniques, a like-Poincar\'{e} inequality
for the marginal Gibbs measure with support on suitable subsets, and
then iterating the argument to recover the required estimate on the
variance.

The proof of Theorem \ref{th:esempio} is given by the explicit
construction of a growing graph with the property of remaining
"expander"  even when the boundary of a finite ball is erased,
as in the free measure. Using a suitable test function,
we will prove the stated exponentially small upper bound
on the spectral gap.

\section{Mixing properties of the plus
phase}\label{sec:mixing}

In this section we analyze the effect of the $(+)$-boundary
condition on the equilibrium properties of the system. In
particular, we prove that the Gibbs measure $\m\equiv\m_{B}^+$
satisfies a kind of \emph{spatial mixing property}, i.e. a form of
weak dependence between spins placed at distant sites.

Before presenting the main result of this section, we need some more
notation and definitions. Recall that for every integer $i$, we
denoted by $B_i=(V_i, E_i)$ the ball of radius $i$ centered in $o$,
and by $B=B_m$ the ball of radius $m$ such that $|V_m|=n$.
Let us define the following objects:
\begin{enumerate}
\item[(i)]  the $i$-th \emph{level}
$L_i = \{x\in V\,: \, d(x,o)=i \}\equiv\partial_V B_{i-1}$;
\item[(ii)] the vertex-set $F_i\subseteq B$ given by $F_i\,:=\,\{x\in
B_{i-1}^c \cap B\}\,$;
\item[(iii)] the $\s$-algebra $\mathcal{F}_i$
generated by the functions $\pi_x$ for $ x\in F_i^c =B_{i-1} $.
\end{enumerate}%
We will be mainly concerned with the Gibbs distribution on $F_i$
with boundary condition $\n\in\Omega^+$, which we will shortly
denote by
$\m_i^\n\,=\,\m_{F_i}^\n\,=\,\m(\cdot\,|\n\in\mathcal{F}_i)$;
analogously we will denote by $\v_i^\n$ the variance w.r.t.
$\m_i^\n$.

Notice that $\{F_i\}_{i=0}^{m+1}$ is a decreasing
sequence of subsets such that $V_m\,=\,F_0\supset
F_1\supset\ldots\supset F_{m+1}\,=\,\emptyset$, and in particular
$\m_i(\m_{i+1}(f))\,=\,\m_{i}(f)$, for all finite $i$,
and $\m_{m+1}(f)=f$. The set of variables $\{\m_{i}(f)\}_{i\geq0}$
is a Martingale with respect to the filtration $\{\mathcal{F}_i\}_{i\geq 0}$.

For a given  $i\in \{0,\ldots,m\}$ and a given subset
$S\subset L_i$, we set $U=F_{i+1}\cup S$
and consider the Gibbs measure
conditioned on the configuration outside $U$ being $\t\in\Omega^+$,
which as usually will be denoted by
$\m_{U}^\t$.

 We are now able to state the following:
\begin{proposition}\label{propdecadimento}
Let $G$ a $(g,o)$-growing graph with maximal degree $\Delta$.
Then there exists a constant $\d=\d(\Delta)>0$ such that,
 for every $\b>\frac{\d}{2g}$, every $\t\in\Omega^+$, and every pair of
vertices $x\in S \subset L_i$ and $y\in L_i\setminus S$, $i\in\{0,\ldots, m\}$,
\be\label{decadimento}
|\m_{U}^\t(\s_x=+)-\m_{U}^{\t^y}(\s_x=+)|\leq c e^{-\b'
d(x,y)}\,,\quad \ee with $\b':=2g\b -\d>0$ and for some constant $c>0$.
\end{proposition}

Let us briefly justify the above result.
Since the boundary of $B$ is proportional to its volume,
the $(+)$-b.c. on $B$ is strong enough
to influence spins at arbitrary distance.
In particular, as we will prove, the effect of the $(+)$-boundary on a given
spin $\s_x$, weakens the influence on $\s_x$ coming from other spins
(placed in vertices arbitrary near to $x$) and gives rise to the decay correlation
stated in Proposition \ref{propdecadimento}. Notice that the
correlation decay increases with $\b$.

The proof of Proposition \ref{propdecadimento} is divided in two parts.
First, we  define a suitable event and show that the correlation
between two spins is controlled by the probability of this event.
Then, in the second part, we estimate this probability using a
Peierls type argument.
Throughout the discussion $c$ will denote a constant which is
independent of $|B|=n$, but may depend on the parameters $\D$ and $g$ of
the graph, and on $\b$. The particular value of $c$ may
change from line to line as the discussion progresses.

\subsection{Proof of Proposition
\ref{propdecadimento}}\label{sec:proofdecay}

 Let us consider two vertices $x\in S \subset L_i$ and $y\in L_i\setminus S$,
 such that $d(x,y)=\ell$, and a configuration $\t\in \Omega^+$.
 Let $\t^{y,+}$ be the configuration that agrees with $\t$ in all sites
but $y$ and has a $(+)$-spin on $y$; define analogously $\t^{y,-}$
and denote by $\m_{U}^{y,+}$ and $\m_{U}^{y,-}$ the measures
conditioned on having respectively  $\t^{y,+}$- and $\t^{y,-}$-b.c..
With this notation and from the obvious fact that the event
$\{\s:\,\s_x=+\}$ is increasing, we get that
\be\label{discrepanza}
|\m_{U}^\t(\s_x=+)-
\m_{U}^{\t^y}(\s_x=+)|\,=\,\m_{U}^{y,+}(\s_x=+)-\m_{U}^{y,-}(\s_x=+)\,.
\ee
In the rest of the proof we will focus on the correlation in the
r.h.s. of (\ref{discrepanza}).

In order to introduce and have a better understanding of the ideas
and techniques that we will use along the proof, we first consider
the case $\ell=1$, which is simpler but with a similar structure
 to the general case $\ell>1$.

\subsubsection{Correlation decay: the case $\ell=1$}
Assume that $\ell=1$, namely that $x$ and $y$ are
neighbors. Denoting by $\m_{U}^{-}$ the measure with
$(-)$-b.c. on  $ U^c=B_i\setminus\{S\} $ and $(+)$-b.c. on
$\partial_V B$, we get
\bea\label{roughbound}
 \m_{U}^{y,+}(\s_x=+)-\m_{U}^{y,-}(\s_x=+)&=&
 \m_{U}^{y,-}(\s_x=-)-\m_{U}^{y,+}(\s_x=-)\nonumber\\
 &\leq&
\m_{U}^{y,-}(\s_x=-)\nonumber\\
&\leq & \m_{U}^{-}(\s_x=-)\,,
\eea
where the last inequality follows by monotonicity. The problem is
thus reduced to estimate the probability of the event
$\{\s:\,\s_x=-\}$ w.r.t. $\m_{U}^{-}$.

Let $\cK$  be the set of connected subsets of $U$ containing $x$
and write
$$
\cK\,=\bigsqcup_{p\geq 1} \cK_p \quad\quad\mbox{ with
}\,\,\, \cK_p =\{C\in\cK\,\mbox{ s.t. }  |C|=p\} \,.
$$
 For any configuration $\s\in\Omega^+$, we denote by $K^{(\s)}$
 the maximal negative component in $\cK$ admitted by $\s$, i.e.
\be\label{def:negcomponent}
K^{(\s)}\in \mathcal K \,\,\, \mbox{
s.t. } \left\{\begin{array}{ll}
\s_z =- &  \forall\, z\in K^{(\s)}\\
\s_z =+ & \forall\, z\in \partial_VK^{(\s)}\cap U
\end{array}\right.\,
\ee
 With this notation the event $\{\s:\,\s_x=-\}$ can be
expressed by means of disjoint events as
\be\label{def:spinnegativo}
\{\s:\,\s_x=-\}\,=\,\bigsqcup_{p\geq 1} \bigsqcup_{C\in\,\mathcal
K_{p}}\{\s:\,K^{(\s)}\,=\,C\}\,,
\ee
and then
 \be\label{misuraspinnegativo}
  \m_{U}^{-}
(\s_x=-)\,=\,\sum_{p\geq 1}\,\,\sum_{C\in\, \mathcal
K_{p}}\m_{U}^{-}( K^{(\s)}=C)\,.
\ee

Let us introduce the symbol $\s\sim C$ for a configuration $\s$ such
that $\s_C=-$ and $\s_{\partial_V C\cap U}=+$.
 The main step in the proof is to show the following claim:
\begin{claim}\label{claim:misuracompon}
If $G$ is  a $(g,o)$-growing graph with maximal degree $\Delta$,
then, for any  subset
$C\subset U$,
 \be\label{misuraC}
 \m_{U}^{-}(\s\sim C)
 \leq e^{-2g\, \b |C| }\,.
 \ee
\end{claim}

The proof of Claim \ref{claim:misuracompon} is postponed to subsection
\ref{subsec:proofclaim}. Let us assume for the moment its validity
and complete the proof of the case $\ell=1$. By Claim
\ref{claim:misuracompon} and from the
definition of $K^{(\s)}$,  we get
 \be
 \m_{U}^{-}( K^{(\s)}=C)\leq e^{-2g\, \b |C| }\,.
 \ee
 We now recall the following Lemma due to Kesten (see
\cite{Kes}).
\begin{lemma}\label{lemma:Kesten} Let $G$ an infinite graph with maximum degree $\Delta$
and let $\mathcal C_p$ be the set of connected sets with $p$
vertices containing a fixed vertex $v$. Then $|\mathcal C_p|\leq
(e(\Delta+1))^p$.
\end{lemma}
\no Applying Lemma \ref{lemma:Kesten} to the set $\cK_p$, we obtain
the bound $|\cK_p|\leq e^{\d p}$, with $\d=1+\log(\Delta+1)$.
Continuing from (\ref{misuraspinnegativo}), we finally get that for
all $\b'=2g\b-\d>0$,
i.e. for all $\b> \frac{\d}{2g}$,
 \bea\label{caso1}
\m_{U}^{-} (\s_x=-)&\leq& \sum_{p\geq 1}
\,\,\sum_{C\in\, \mathcal K_{p}}
e^{-2g\, \b p}\nonumber\\
&\leq& \sum_{p\geq 1} e^{-2g\b p} e^{\d p}\nonumber\\
&\leq& c e^{-\b'}
\eea
 which concludes the proof of (\ref{propdecadimento}) in the case
 $\ell=1$.

Notice that the argument above only involves the spin at $x$, and
thus applies for all pairs of $x,y\in L_i$, independently of
their distance. Anyway,  when $d(x,y)>1$ this method does
not provide the decay with the distance stated in Proposition
\ref{propdecadimento}, and  a different approach is required.

\subsubsection{Correlation decay:  the case $\ell>1$}\label{subsect:corrdecaygeneral}
Let us now consider two vertices $x\in S \subset L_i$ and
$y\in L_i\setminus S$, such that
$d(x,y)=\ell>1$.  Before defining new objects, we want to clarify
the main idea beyond the proof. Since the measure $\m^\t_{U}$
fixes the configuration on all sites in $U^c\equiv
B_i\setminus S$, the vertex $y$ can communicate with $x$ only
through paths going from $x$ to $y$ and crossing vertices in $U$.
However, the effect of this communication can be very small compared
to the information arriving to $x$ from the $(+)$-boundary. In
particular, if every path starting from $y$ crosses a $(+)$-spin
before arriving to $x$, then the communication between them is
interrupted. Let us formalize this assertion.

We denote by  $\mathcal C$ the set of connected subsets $C\subseteq
U\cup \{y\}$ such that $y\in C$, and call an element $C\in\mathcal
C$ a \emph{component} of $y$.
For every configuration $\s\in\Omega^+$, we define $C^{(\s)}$ as the
maximal component of $y$ which is negative on $C^{(\s)}\cap U$,
i.e
 \be\label{def:negcomponent2}
 C^{(\s)}\in
\mathcal C \,\,\, \mbox{ s.t. } \left\{
\begin{array}{ll}
\s_z =- &  \forall\, z\in C^{(\s)}\cap U\\
\s_z =+ & \forall\, z\in \partial_VC^{(\s)}\cap U
\end{array}\right.\,.
\ee
Observe that the spin on $y$ is not fixed under the event
 $\{\s:\,C^{(\s)}=C\}$.
 Finally, let
 $\mathcal C^\emptyset :=\{C\in\mathcal C\,\mbox{ s.t. }\, x\not\in C\}$
and define the event
\be\label{def:eventA}
A\,:=\,\{\s:\,
C^{(\s)}\in\cC^\emptyset\}\,=\,\bigsqcup_{C\in\cC^\emptyset}\{\s:\,
C^{(\s)}=C\}\,.
\ee
Then we have
\begin{eqnarray}\label{condizioneA}
\m_{U}^{y,-}(\s_{x}=+\,|\,A)&=& \sum_{C\in\mathcal
C^\emptyset}\m_{U}^{y,-}(\s_{x}=+,C^{(\s)}
=C\,|\,A)\nonumber\\
&=& \frac{\sum_{C\in\mathcal
C^\emptyset}\m_{U}^{y,-}(\s_{x}=+,C^{(\s)} =C\,)}{\sum_{C\in
\mathcal
C^\emptyset}\m_{U}^{y,-}(\,C^{(\s)} =C\,)}\nonumber\\
&=& \frac{\sum_{C\in \mathcal
C^\emptyset}\m_{U}^{y,-}(\s_{x}=+\,|\,C^{(\s)}
=C\,)\m_{U}^{y,-}(\,C^{(\s)}=C\,)}{\sum_{C\in \mathcal
C^\emptyset}\m_{U}^{y,-}
(\,C^{(\s)} =C\,)}\nonumber\\
&\geq& \min_{C\in \mathcal
C^\emptyset}\m_{U}^{y,-}(\s_{x}=+\,|\,C^{(\s)}=C)\,.
\end{eqnarray}

 Notice that when the measure
 $\m_{U}^{y,-}$ is conditioned on the event
 $\{\s:\,C^{(\s)}=C\}$, the spin
 configuration on $\partial_V C$ is completely  determined by the boundary
 condition:
  on $\partial_VC\cap U$ it is given by all
 $(+)$-spins and on $\partial_VC\cap U^c$ it corresponds to $\t^{y,-}$.
 Hence,  spins on  $U\setminus (C\cup\partial_V C)$
 become independent of spins on $C$, and  we get
\bea\label{stocdom}
\m_{U}^{y,-}(\,\cdot\,|\,C^{(\s)}=C)&=& \m_{K_x}^{y,-}(\,\cdot\,|
\s_z=+\,,\,z\in \partial_VC \cap U)
\nonumber\\
&=& \m_{U}^{y,+}(\,\cdot\,| \s_z=+\,,\,z\in (C\cup\partial_VC)\cap U)
\nonumber\\
&\geq& \m_{U}^{y,+}(\,\cdot\,)\,,\eea
 where the last inequality follows by
stochastic domination. Being $\{\s:\,\s_x=+\}$ an increasing event,
and from (\ref{condizioneA}) and (\ref{stocdom}),
  we get
$$
\m_{U}^{y,-}(\s_{x}=+\,|\,A)\geq \m_{U}^{y,+}(\s_{x}=+)\,,
$$
 which with the obvious fact that
$\m_{U}^{y,-}(\s_{x}=+)\,\geq
\m_{U}^{y,-}(\s_{x}=+\,|\,A)\,\m_{U}^{y,-}(A)\, $, implies
\be\label{path}
\m_{U}^{y,+}(\s_{x}=+)-\m_{U}^{y,-}(\s_{x}=+)
\,\leq\, \m_{U}^{y,-}(A^c) \,.
\ee
By monotonicity and being  $A^c$ a decreasing event, we get
the inequality $\m_{U}^{y,-} (A^c)\leq \m_{U}^{-} (A^c)\,$, where, we recall,
$\m_{U}^{-}$ denotes the measure on $U$ conditioned on
having all $(-)$-spins on $U^c$. We now focus on $\m_{U}^{-}
(A^c)$.\vspace{3mm}

Let $\mathcal C^{\ne\emptyset}$  denote the set of components of $y$
containing $x$, and for every $p\in\N$, let
$\mathcal C_p$ be the set of components in $\mathcal C^{\ne\emptyset}$ with
$p$ vertices, i.e
$$
\mathcal C_{p} :=\{C\in\mathcal C^{\ne\emptyset} \,\mbox{ s.t. }\, |C|=p \}
\,\quad\quad \mathcal C^{\ne\emptyset} := \bigsqcup_{p> 0} \mathcal C_p.
$$
Notice that if $C\in\mathcal C^{\ne\emptyset}$, then
$|C|\geq \ell +1$, since $d(x,y)=\ell$.
Thus, $A^c$
can be expressed by means of disjoint events as
\be\label{def:eventAC}
A^c\,=\,\bigsqcup_{p\geq \ell+1} \bigsqcup_{C\in\,\mathcal
C_{p}}\{\s:\,C^{(\s)}\,=\,C\}\,,
\ee
 and we get
\be\label{misuraAc}
 \m_{U}^{-} (A^c)\,=\,\sum_{p\geq
\ell +1}\,\,\sum_{C\in\, \mathcal C_{p}}\m_{U}^{-}(
C^{(\s)}=C)\,. \ee
Since $\partial_V (C\setminus\{y\})\cap U \subseteq \partial_V C\cap U$,
we observe that the event $\{\s: C^{(\s)}=C\}\equiv\{\s:\,
\s_{C\setminus\{y\} }=-\,,\,
\s_{\partial_V C\cap U}=+ \}$ is a subset of $\{\s:\,
\s_{C\setminus\{y\}}=-\,,\,\s_{\partial_V (C\setminus\{y\})\cap U}=+ \}
\equiv\{\s:\s\sim C\setminus\{y\} \}$.
Applying  the
result stated in Claim \ref{claim:misuracompon} to the set
$C\setminus\{y\} $, we obtain the bound
 \be\label{misuraC2}
 \m_{U}^{-}(C^{(\s)}=C)\leq e^{-2g\, \b (|C|-1) }\,,
 \ee
 which holds under the same hypothesis of the claim.
Continuing from (\ref{misuraAc}), we
then have that for all $\b'=2g\b-\d>0$, i.e. for all $\b> \frac{\d}{2g}$,
\bea\label{caso2}
\m_{U}^{-} (A^c)&\leq& \sum_{p\geq
\ell+1}\,\,\sum_{C\in\, \mathcal
C_{p}}e^{-2g\, \b (p-1)}\nonumber\\
&\leq& e^{\d} \sum_{p\geq \ell} e^{-(2g \b-\d) p} \nonumber\\
&\leq& c e^{-\b'\ell},
\eea
where in the second line we used the
bound $|\mathcal C_{p}|\leq e^{\d p}$ due to Lemma \ref{lemma:Kesten}.
This concludes the proof of Proposition \ref{propdecadimento}.
In the next subsection we will go back and prove
Claim \ref{claim:misuracompon}.

\subsection{Proof of Claim
\ref{claim:misuracompon}}\label{subsec:proofclaim}
 To estimate the probability $\m_{U}^{-}(\s\sim C )$, we
now appeal to a kind of Peierls argument that runs as
 follows (see also \cite{JS}).
Given a subset $C\subseteq U$,  we  consider the edge
boundary $\partial_E C$  and define

\be\label{def:boundaries}
\begin{array}{l}
\partial_+
C\,:=\,\{e=(z,w)\in \partial_E C\,:\,z, \,w \in U \}\\
\partial_- C\,:=\,\{e=(z,w)\in \partial_E C\,:\,z \mbox{ or } w \in
U^c \}\end{array}.
\ee
The meaning of this
notation can be better understood if we consider a configuration
$\s\in\Omega_{U}^-$ such that $C^{(\s)}=C$ (see
(\ref{def:negcomponent2})). In this case $\s$ has $(-)$-spins on
both the end-vertices of every edge in $\partial_- C$ and a
$(+)$-spin in one end-vertex of every edge in $\partial_+ C$.
Similarly if we consider $\s$ such that $K^{(\s)}=C$ (see
(\ref{def:negcomponent})).

 For every $\s\in\Omega_{U}^{-}$ such that
$\s\sim C$, let $\s^*\in \Omega_{U}^{-}$ denote the configuration
obtained by a global spin flip of $\s$ on the subset $C$, and
observe that the map $\s\rightarrow\s^*$ is injective. This flipping
changes the Hamiltonian contribute of the interactions just along
the edges in $\partial_E C$. In particular $\s^*$ loses the positive
contribute of the edges in $\partial_+ C$ and gains the contribute
of the edges in $\partial_- C$, and then we get
\be\label{hamilt}
H_{U}^-(\s^*)\,=\,H_{U}^-(\s)-2(|\partial_+ C
|-|\partial_- C|)\,.
\ee
 From this, we have
 \bea\label{mispathlunghi}
 \m_{U}^-(\s\sim C)&=&\sum_{\{\s: \s\sim C\}}
\frac{e^{-\b H_{U}^-(\s)}}{Z_{U}^-}\nonumber\\
&\leq& \frac{\sum_{\{\s: \s\sim C\}}e^{-\b
H_{U}^-(\s)}}{\sum_{\{\s: \s\sim C\}}e^{-\b H_{U}^-(\s^*)}}\nonumber\\
&=& e^{-2\b(|\partial_+
C|-|\partial_- C|)}\,,
\eea
 where in the first inequality we reduced the partition function to a summation
 over
$\{\s: \s\sim C\}$ and then we applied (\ref{hamilt}).

The following Lemma concludes the proof of Claim \ref{claim:misuracompon}.
\begin{lemma}\label{lemma:peierls}
Let $G$ a $(g,o)$-growing graph with maximal degree $\Delta$.
Then, for every subset $C\subseteq U$,
 \be\label{differenza}
 |\partial_+ C|-|\partial_- C| \geq g |C|\,. \ee
\end{lemma}

\begin{proof}
For a subset $C\subseteq U $, we define the
downward boundary of $C$, $\partial_\downarrow C$, as the edges
of $\partial_E C$ such that the endpoint in $C$ is in a higher level
(strictly small index) than the endpoint not in $C$, i.e.
$$
\partial_\downarrow C =\{(u,v)\in \partial_E C:
\exists j \mbox{ s.t. } u\in L_j\cap C,\,v\in{L_{j+1}}\}\,.
$$
We then define the not-downward boundary of $C$, $\partial_\Rsh C$,  as the edges in
$\partial_E C$ which are not-downward edges, i.e.
$\partial_\Rsh C= \partial_E C\setminus \partial_\downarrow C$.\\
Notice that $\partial_\downarrow C\subseteq \partial_+ C$, while
$\partial_\Rsh C\supseteq \partial_- C$. In particular, inequality (\ref{differenza})
follows from the bound
\be\label{bordi}
|\partial_\downarrow C|-|\partial_\Rsh C|\geq g|C|\,.
\ee
For all $j\geq 0$, define  $C_j= C\cap L_j$  and notice that,
by the growing property of $G$,
\be\label{bordiparziali}
|\partial_\downarrow C_j|-|\partial_\Rsh C_j|\geq g|C_j|\,.
\ee
Moreover, one can easily realize that
\be\label{telescopica}
|\partial_\downarrow C|-|\partial_\Rsh C|= \sum_{j\geq 0}
|\partial_\downarrow C_j|-|\partial_\Rsh C_j|\,.
\ee
In fact, all edges belonging to  $\partial_E C_j$, for some $j$,
but not belonging to $\partial_E C$, are summed once as downward edges
and  subtracted once as non-downward edges.
In conclusion, the last two inequalities imply bound (\ref{bordi})
and conclude the proof of the lemma.
\end{proof}

\section{Fast mixing inside the plus phase}
In this section we will prove  that the spectral gap of the Glauber
dynamics, in the situation described by Theorem \ref{teo}, is
bounded from zero uniformly in the size of the system. From
Definition \ref{cgap} of spectral gap, this is equivalent to showing
that for all inverse temperature $\b\gg1$, the Poincar\'{e}
inequality
$$
\var(f)\leq c\, \mathcal{D}(f)\,,\quad \forall f\in L^2(\Omega^+,\cF,\m)
$$
holds with constant $c$  independent of the size of $B$.

First, we  give a brief sketch of the proof. The rest of the section
is divided into two parts. In the first part, from the mixing
property deduced in section \ref{sec:mixing}
 and by means of coupling techniques, we
derive a Poincar\'{e} inequality for some suitable marginal Gibbs
measures. Then, in the second part, we will run a recursive argument
that together with some estimates, also derived from Proposition
\ref{propdecadimento}, will yield the Poincar\'{e} inequality for
the global Gibbs measure $\m$.

\subsection{Plan of the Proof}\label{subsec:sketchproof}

Let us first recall the following decomposition property of the
variance which holds for all subsets $D\subseteq C \subseteq B$,
\be\label{pro}
\v_C ^{\n}(f)\,=\,\m_C^{\n}[\v_D(f)]+\v_C ^{\n}[\m_D(f)]\;.
\ee
Applying recursively (\ref{pro}) to subsets $B\equiv
F_0\supset F_1\supset\ldots\supset F_{m+1}=\emptyset$ and recalling
the relations $\m_i(\m_{i+1}(f))\,=\,\m_{i}(f)$ and
$\m_{m+1}(f)=f$, we obtain
\begin{eqnarray}\label{ric1}
\v(f)&=& \m[\v_m(f)]+\v[\m_m(f)] \nonumber\\ &=&
\m[\v_m(\m_{m+1}(f))]+
\m[\v_{m-1}(\m_m(f))]+ \v[\m_{m-1}(\m_m(f))] \nonumber\\
&=& \vdots \nonumber\\ &=&
\sum_{i=0}^{m}\mu[\v_{i}(\m_{i+1}(f))]\,.
\end{eqnarray}
Notice that (\ref{ric1}) can also be seen as a decomposition
of the Martingale given by the set of variables
$\{\m_{i}(f)\}_{i\geq 0}$ respect to the filtration
$\{\mathcal{F}_{i}\}_{i\geq 0}$.

To simplify the notation we define $g_i\,:=\,\m_i(f)$ for all
$i=0,\ldots, m+1$. Notice that $g_i \in \cF_i$. Inserting $g_i$ in
(\ref{ric1}), we then have that
\be\label{ric2}
\v(f)\,=\,\sum_{i=0}^{m}\mu[\v_i(g_{i+1})]\,.
\ee

The proof of the Poincar\'{e} inequality for $\m$, with constant
independent of the size of the system, is given in the following two
steps:
\begin{enumerate}
\item Proving that  $\forall\,\t\in\Omega^+$ and
$i\in\{0,\ldots,m\}$, there exist suitable vertex-subsets $\{K_x\}_{x\in L_i}$,
$K_x\ni x$, such that the like-Poincar\'{e} inequality
\be\label{uno}
\v_{i}^\t(g_{i+1})\leq c\,\sum_{x\in L_i}\m_i^\t(\v_{K_x}
(g_{i+1}))
\ee
holds with
constant $c$ uniformly bounded in the size of $L_i$;
 \item Relating the  variance of $g_i=\m_i(f)$ to
the  variance of $f$ in order to get an inequality of the kind
\be\label{due}
\sum_{i=0}^m\sum_{x\in L_i}\m(\v_{K_x} (g_{i+1}))\leq c
\cD(f)+\varepsilon \sum_{i=0}^m\sum_{x\in L_i}\m(\v_{K_x} (g_{i+1}))
\ee
 with $\varepsilon$ a small quantity for $\b\gg 1$.
\end{enumerate}

Notice that from (\ref{due}) the inequality
$$
\sum_{i=0}^m\sum_{x\in L_i}\m(\v_{K_x} (g_{i+1}))\leq c(1-\varepsilon)^{-1}\cD(f)\,
$$
follows with $c(1-\varepsilon)^{-1}=\Omega(1)$ for all $\b\gg 1$. Together
with Eqs. (\ref{ric2}) and (\ref{uno}), this will establish the
required Poincar\'{e} inequality for $\m$ and therefore will
conclude the proof of Theorem \ref{teo}.

\subsection{ Step 1: From correlation decay to Poincar\'{e} inequality}
 In this section we prove that under the same
hypothesis of Proposition \ref{propdecadimento},  the marginal
of the conditioned Gibbs measure on some suitable  subsets,  satisfies a
Poincar\'{e} inequality with constant independent of the size of these
subsets.

To state the result, let us fix a subset $S\subseteq L_i$
and a configuration  $\t\in\Omega^+$. We then define the measure
\be\label{def:margmeas}
\nu_{S}^\t(\s)\,:=\,\sum_{\n: \n_S=\s_S}
\m(\n\,|\,\t\in \mathcal F_{B_{i}\setminus S})\,,
\ee
which is  the
marginal of the Gibbs measure $\m^\t_{F_{i+1}\cup S}$  on $S$,
and denote by $\v_{\nu_S^\t}$ the variance w.r.t. $\nu_S^\t$.
We state the following:
\begin{theorem}\label{teoPoin}
For all $\b\gg 1$ and for
every subset $S\subseteq L_i$, $\t\in \Omega^+$ and $f\in
L^2(\Omega, \cF_S, \nu_{S}^\t)$, the measure $\nu_{S}^\t$ satisfies
the Poincar\'{e} inequality
\be\label{poin}
\v_{\,\nu_{S}^\t}(f)\leq
c_0\sum_{x\in S} \nu_{S}^\t(\v_x (f))\,.
\ee
with
$c_0=c_0(\b, \Delta, g)=1+O(e^{-c\b})$.
\end{theorem}

\begin{remark}\label{remark}
Before proceeding  with the proof of Theorem \ref{teoPoin}, we point
out that this result includes, as a particular case, inequality
(\ref{uno}) for subsets $K_x= F_{i+1}\cup \{x\}$.
 To see that, choose $S=L_i$ so that
 $\m(\cdot\,|\mathcal F_{B_{i}\setminus S})\equiv\m_i$.
An easy computation shows that, for every function $f\in\mathcal F_{i+1}$,
 $\nu_{L_i}^\t(f)\equiv\m_i^\t(f)$ and
 $\nu_{L_i}^\t(\v_x(f))= \m_i^\t(\v_{K_x} (f))$.
Inequality (\ref{poin}) then corresponds
to the like-Poincar\'{e} inequality
$$
\v_{i}^\t(g_{i+1})\leq c_0\,\sum_{x\in
L_i}\m_i^\t(\v_{K_x} (g_{i+1}))\,,
$$
which concludes the first step of the proof of
Theorem \ref{teo}.
\end{remark}

\subsubsection{Proof of Theorem \ref{teoPoin}}
The proof of Theorem \ref{teoPoin} rests on the so called
\emph{coupling technique}. This is a useful method to bound from above
the mixing time of Markov processes,  introduced for the first time
in this setting by Aldous \cite{Al} and subsequently refined to the
\emph{path coupling} \cite{BD, LRS}. See also \cite{Lin} for a wider
discussion on the coupling method.

 A \emph{coupling} of two measure
$\m_1$ and $\m_2$ on $\Omega$ is any joint distribution $\rho$ on
$\Omega\times\Omega$ whose marginal are $\m_1$ and $\m_2$
respectively. Here, we want to construct a coupling of two Glauber
dynamics on $\Omega_S$ with same reversible measure $\nu_S^\t$ but
different initial configurations.  We denote by $\mathcal L_S$ the
generator of this dynamics.  We also recall that for all
$\n\in\Omega_S^\t$, $x\in S$ and  $a\in\{\pm 1\}$, the jump rates of
the heat bath version of the dynamics (see Def.
\ref{rates} ) are given by
 \bea\label{rate2}
  c_x(\n, a)&=&\,
  \nu^\t_{S}(\s_x=a\,|\,\n\in \mathcal{F}_{S\backslash x})\nonumber\\
  &=& \m(\s_x=a\,|\,\n\in \mathcal{F}_{S\backslash x},
\t\in\mathcal F_{B_i\setminus S} )\nonumber\\
&=& \m_{K_x}^\n (\s_x=a)\,,
\eea
 where in the second line we applied
the definition  of $\nu_S^\t$  and used that
$\{\s_x=a\}\in \mathcal{F}_S$, and in the last line we
adopt the notation $K_x:=F_{i+1}\cup \{x\}$.

We now consider the coupled process $(\n(t),\xi(t))_{t\geq 0}$ on
$\Omega_S\times\Omega_S$ defined as follows. Given the  initial
configurations $(\n,\xi)$, we let the two dynamics evolve at the same
time and update the configurations at the same vertex. We then chose
the coupling jump rates $\tilde{c}_x ((\n,a),(\xi,b))$  to go from
$(\n,\xi)$ to $(\n^{x,a}, \xi^{x,b})$, with $a,b\in\{\pm 1\}$, as the
optimal coupling (see \cite{Lin}) between the jump rates
$\m_{K_x}^\n(\s_x=a)$ and $\m_{K_x}^\xi(\s_x=b)$. More explicitly,
for $a\in\{\pm 1\}$, they are given by %
\be\label{couplingrates}
\left\{
\begin{array}{l}
\tilde{c}_x((\n,a),
(\xi,a))=\min\{\m_{K_x}^\n(\s_x=a)\,;\,\m_{K_x}^\xi(\s_x=a)\}\hspace{2.5cm}\\
\tilde{c}_x((\n,a), (\xi,-a))=\max \{0\,;\,
\m_{K_x}^\n(\s_x=a)-\m_{K_x}^\xi(\s_x=a)\}\hspace{2.5cm}
\end{array}
\right.
\ee
We denote by $\widetilde{\cL}$ the generator of the coupled
process, and by $\widetilde P_t $ the correspondent Markov
semigroup. Notice that from our choice of coupling jump rates,
we get that the probability of
disagreement in $x$, after one update in $x$ of $(\n,\xi)$, is given by%
\be\label{def:probdisagree}
P_{dis}^{^x}(\n,\xi):=
|\m_{K_x}^\n(\s_x=+)-\m_{K_x}^\xi(\s_x=+)|\,.
\ee

 Let us now consider  the subset
$H\subset\Omega_S\times\Omega_S$ given by all couples of
configurations which differ by a single spin flip in some vertex of
$S$. One can easily verify that the graph
$(\Omega_S,\,H)$ is connected and that the induced
graph distance between configurations
$(\n,\xi)\in\Omega_S\times\Omega_S$,  $D(\n,\xi)$,
just corresponds to their Hamming
distance. Let us also  denote by $\mathbb E_{\n,\xi}[D(\n(t),
\xi(t))]$ the average
distance at time $t$ between two coupled configurations of the
process starting at $(\n,\xi)$.
 We claim
the following:
\begin{claim}\label{claim:coupling}
For all $\b\gg 1$, there exists
a positive constant $\a\equiv\a(\b,g,\Delta)$ such that, for  every initial
configurations $(\n,\xi)\in H$, the process
$(\n(t),\xi(t))_{t\geq0}$ satisfies the inequality \be
\frac{d}{dt}\,\mathbb
E_{\n,\xi}[D(\n(t),\xi(t))]\left|_{t=0}\right.\leq -\,\a \,.\ee
\end{claim}

\begin{proof}[Proof of Claim \ref{claim:coupling}]
The derivative in $t$ of the average distance, computed for $t=0$,
can be written as
\bea\label{distanzamedia}
&&\frac{d}{dt}\,\mathbb
E_{\n,\xi}[D(\n(t),\xi(t))]\left|_{t=0}\right.=
\frac{d}{dt}\,\left(\widetilde P_t D
\right)(\n,\xi)\left|_{t=0}\right.\,=\, (\widetilde \cL \,D) (\n,\xi)
\nonumber \\
&&\quad\quad\quad\quad= \sum_{x\in
S}\,\,\sum_{a,b\in\{\pm 1\}} \tilde
c_x((\n,a)(\xi,b))[D(\n^{x,a},\xi^{x,b})-D(\n,\xi)]\,.
\eea

Since $(\n,\xi)\in H$, there exists a vertex $y\in S$ such that
$\xi=\n^y$.  If $x=y$, then  $P_{dis}^{^x}(\n,\n^x)=0$ and the
distance between the updated configurations decreases of one. While
if $x\ne y$, with probability $P_{dis}^{^x}(\n,\n^y)$  the updated
configurations have different spin at $x$ and  their distance
increases by one.
Continuing from (\ref{distanzamedia}) we get, for all $\b\gg 1$,
\bea\label{averagedistance}
\frac{d}{dt}\,\mathbb E_{\n,\xi}[D(\n(t),\xi(t))]\left|_{t=0}\right.
&=&
-1+ \sum_{x\in S \atop x\ne y}{P_{dis}^{^x}(\n,\n^y)} \nonumber \\
&\leq&
-1+ c \sum_{\ell\geq1} e^{-\b'\ell}\Delta^\ell\nonumber\\
&\leq&
 -(1-c e^{-\b'})\,,
\eea
where in the second line we used the bound
\be\label{pdisagreestima}
P_{dis}^{^x}(\n,\n^y)=|\m_{K_x}^\n(\s_x=+)-\m_{K_x}^{\n^y}(\s_x=+)|
\leq c e^{-\b' d(x,y)}\,,
\ee
 which holds for all $\b'=\d\b-2g>0$ as stated in Proposition
 \ref{propdecadimento}.
 Claim \ref{claim:coupling} follows
taking $\a=(1-c e^{-\b'})$ and
$\b$ sufficiently large.
\end{proof}

Using the \emph{path coupling} technique (see \cite{BD})  we can
extend the result of Claim \ref{claim:coupling}  to arbitrary
initial configurations $(\n,\xi)\in\Omega_S\times\Omega_S$, and obtain
\be\label{pathcoupling}
\frac{d}{dt}\,\mathbb
E_{\n,\xi}[D(\n(t),\xi(t))]\left|_{t=0}\right.\leq -\,\a
D(\n,\xi)\,.
\ee
From (\ref{pathcoupling}) it  now follows
straightforwardly  that $\E_{\n,\xi}[D(\n(t),\xi(t))]\leq e^{-\a\,t}
D(\n,\xi)$, and then we get
\be\label{disagreement}
\P(\n(t)\ne
\xi(t))\leq \mathbb E_{\n,\xi}(D(\n(t),\xi(t)))\leq e^{-\a\,t}
D(\n,\xi)\,.
\ee
To bound the spectral gap $\cg (\nu_S^\t)$ of the
dynamics on $S$, we consider an eigenfunction $f$ of $\mathcal
L_S$ with eigenvalue $-\cg(\nu_S^\t)$, so that
$$
\mathbb E_\s f(\n(t))= e^{t \mathcal L_S } f(\n)\,=\,
 e^{-\cg(\nu_S^\t)\, t}f(\n)\,.
$$
Since the identity function has eigenvalue zero, and  therefore it is
orthogonal to $f$, it follows that $\nu_S^\t(f)=0$  and
$\nu_S^\t(\mathbb E_\xi f(\xi(t)))=0$, where $\nu_S^\t$ is the
invariant measure for $\mathcal L_S$. From these considerations and
inequality (\ref{disagreement}), we have
\bea
e^{t \mathcal L_S }f(\s)
&=&
\mathbb E_\n f(\n(t))- \nu_S^\t(\mathbb E_\xi f(\xi(t)))\nonumber\\
&=&\sum_{\xi} \nu_S^\t(\xi)[\,E_\n f(\n(t))- E_\xi
f(\xi(t))\,]\nonumber\\
&\leq&
2\|f\|_{\infty} \sup_{\n,\xi}\P(\n(t) \ne\xi(t))\nonumber\\
&\leq&
2\|f\|_{\infty}|S| e^{-\a\,t}\,.
\eea
From the last
computation, which holds for all $\n\in\Omega_S^\t$ and for all $t$,
we finally obtain that $\cg (\nu_S^\t)\geq \a$ independently of the
size of $S$, which implies the Poincar\'{e} inequality (\ref{poin})
with constant $c_0=\a^{-1}= 1+O(e^{-c\b})$. This concludes the proof
of Theorem \ref{teoPoin}. {\hfill$\square$\vskip.5cm}

\subsection{Step 2: Poincar\'{e} inequality for the global Gibbs measure}
With the previous analysis we obtained a like-Poincar\'{e} inequality for
the marginal of the measure $\m_i$ on  the level $L_i$ (see Remark \ref{remark}), which
inserted in  formula (\ref{ric2}) provides the bound%
 \be\label{pvar}
\v(f)\,\leq\, c_0\sum_{i=0}^m\sum_{x\in L_i}\m\left[\m_i(\v_{K_x}
(g_{i+1}))\right]\,.
\ee
 Using the same notation as in \cite{MSW},
let us denote the sum in the r.h.s. of (\ref{pvar}) by $\pv(f)$. The
aim of the following analysis is to study $\pv(f)$ in order to
find an inequality of the kind $\pv(f)\leq c\mathcal{D}(f)+ \varepsilon\pv(f)
$, with $\varepsilon=\varepsilon(\b, g, \Delta)<1$ independent of the size of
the system. This would imply that
$$
\v (f)\leq c_0\cdot \pv (f)\leq
c_0 \frac{c}{1-\varepsilon}\mathcal{D}(f)\,,
$$
and then would conclude the proof of Theorem \ref{teo}.

 In this last part of the section, we will first relate the local variance of
 $g_i=\m_i(f)$ with the local variance of $f$. This will produce a
 covariance term that will be analyzed using  a recursive
 argument.

\subsubsection{Reduction to covariance} In order to reconstruct
the Dirichlet form of $f$ from (\ref{pvar}),
 we want to extract the local variance of $f$ from the local
variance of $g_{i+1}$. Notice that w.r.t.
the measure $\m_{K_x}$, the function $g_{i+1}$ just depends on $x$.
Fixing $x\in L_i$ and $\t\in\Omega^+$, and defining
$p(\t):=\m_{K_x}^\t (\s_x=+)$ and $q(\t):=\m_{K_x}^\t (\s_x=-)$, we can write%
\be\label{relatingvariance}
\m_{i}\left(\v_{K_x}
(g_{i+1})\right)\,=\,\sum_{\t} \m_{i}(\t)p(\t)q(\t)\,(\nabla_x\,
g_{i+1}(\t))^2\,.
\ee
Using the martingale property $g_{i+1}= \m_{i+1}(g_{i+2})$, the
variance $\v_{K_x} (g_{i+1})$ can be split in two terms, stressing
the dependence on $x$ of  $g_{i+2}$ and of the conditioned measure
$\m_{i+1}$. Let us formalize this idea.

 For a given configuration $\t \in \Omega^+$ we introduce the
symbols
$$
\t^+\,:=\,\left\{\begin{array}{c} \t^+_y=\t_y\quad \mbox{ if } y\ne x\\
\t^+ _y=+\quad\mbox{ if } y=x \end{array}\right.\qquad
\t^-\,:=\,\left\{\begin{array}{c} \t^-_y=\t_y\quad \mbox{ if } y\ne x\\
\t^- _y=-\quad\mbox{ if } y=x \end{array}\right.\,
$$
and  define the density
\be\label{def:density}
h_x(\s):=\frac{\m_{i+1}^{\t^+}(\s)}{\m_{i+1}^{\t^-}(\s)}\,,\quad\mbox{
with }\quad \m_{i+1}^{\t^-}(h_x)=1\,.
\ee
Whit this notation and continuing from (\ref{relatingvariance}),
we get
\bea
 \m_{i}(\v_{K_x} (g_{i+1}))
 &=&
\sum_{\t}\m_{i}(\t)p(\t)q(\t)\left[\nabla_x\,\m_{i+1}(g_{i+2})(\t)
\right]^2\hspace{5cm}\nonumber \\
&=&
\sum_{\t}\m_{i}(\t)p(\t)q(\t)
\left[\m_{i+1}^{\t^-}(g_{i+2})-\m_{i+1}^{\t^+}(g_{i+2})\right]^2\nonumber\\
& = &
\sum_{\t}\m_{i}(\t)p(\t)q(\t)\left[ \m_{i+1}^{\t^+}(\nabla_x
g_{i+2})-\m_{i+1}^{\t^-}(h_x, g_{i+2})\right]^2\,\nonumber \\
\label{1termine}
&\leq&
2\sum_{\t}\m_{i}(\t)p(\t)q(\t)\left[
\left(\m_{i+1}^{\t^+}(\nabla_x
g_{i+2})\right)^2+\left(\m_{i+1}^{\t^-}(h_x,
g_{i+2})\right)^2\right]
\eea

Consider the first term of (\ref{1termine}) and notice that
$ \m_{i+1}^{\t^+}(\nabla_x g_{i+2}) =
\m_{i+1}^{\t^+}(\nabla_x f)$.
To understand this fact, it is enough
to observe that the dependence on $x$ of $g_{i+2}=\m_{i+2}(f)$ comes
only from $f$, since the b.c. on $B_{i+1}$ is fixed equal to
$\t^+$. Replacing $\nabla_x g_{i+2}$ by $\nabla_x f$ and applying  the
Jensen inequality, we get
\be\label{primo}
\sum_{\t}\m_{i}(\t)
p(\t)q(\t)\left(\m_{i+1}^{\t^+}(\nabla_x g_{i+2})\right)^2\leq
\sum_{\t}\m_{i}(\t)p(\t)q(\t)\left(\m_{i+1}^{\t^+}
\left(\nabla_x f\right)^2\right)\,.
\ee
Now, notice that
$p(\t)\m_{i+1}^{\t^+}(\s^+)=\m_{K_x}^\t(\s^+)=
\m_x^\s(\s_x=+)\left(\m_{K_x}^\t(\s^+)+\m_{K_x}^\t(\s^-)\right)$.
This, together with the fact that $\nabla_x f$ does not depend on $x$,
means that the expression on the r.h.s. of (\ref{primo}) equals to
\be\label{primo1}
\sum_{\t}\m_{i}(\t)q(\t)
\sum_\s \m_{K_x}^\t(\s)\m_{x}^\s(\s_x=+)\left(\nabla_x f(\s)\right)^2
\leq
\sum_{\t}\m_{i}(\t)\inf_{\s\in\Omega_{K_x}^\t}
\left\{\frac{q(\t)}{\m_x^\s(\s_x=-)}\right\}\,.
\ee
Since $\s$ agrees with $\t$ on $K_x^c$, then
\be\label{ratio}
\frac{q(\t)}{\m_x^\s(\s_x=-)}=\frac{\m_{K_x}^\t(\s_x=-)}
{\m_{K_x}^\t(\s_x=-|\s)}
= \frac{\m_{K_x}^\t(\s)}{\m_{K_x}^\t(\s|\s_x=-)}\leq h_x(\s)\,,
\ee
and hence,
$\inf_\s \left\{\sfrac{q(\t)}{\m_x^\s(\s_x=-)}\right\}\leq \|h_x\|_\infty$.

To bound $\|h_x\|_\infty$, we first write
\be
h_x(\s) =\frac{\m_{k_x}^{\t}(\s|\s_x=+)}{\m_{k_x}^{\t}(\s|\s_x=-)}
=\frac{\m_{k_x}^{\t}(\s_x=+|\s)}{\m_{k_x}^{\t}(\s_x=-|\s)}\cdot
\frac{\m_{k_x}^{\t}(\s_x=-)}{\m_{k_x}^{\t}(\s_x=+)}\,.
\ee
Taking the supremum  over $\s$ of $h_x$,
we then have
\be\label{normah_x}
\|h_x\|_\infty \leq \frac{1}{\m_{i+1}^{\t^-}
(\s_y=+; \forall y\in  N_x\cap L_{i+1})}
\leq \frac{1}{1-\sum_{y\in N_x\cap L_{i+1}}\m_{i+1}^{-}(\s_y=-)}\,,
\ee
where we denoted by $\m_{i+1}^{-}$ the measure conditioned on
having all minus spins in $B_{i}$ and plus spins in $\partial_V B$.
Inequality (\ref{caso1}), applied to $U= F_{i+1}$, implies that
\be\label{spin-}
\m_{i+1}^{-}(\s_y=-)\leq c e^{-\b'} \,,
\ee
 with
$\b'=2g\b-\d$ as in Proposition \ref{propdecadimento}.
 Combining (\ref{normah_x}) and (\ref{spin-}),
 we get that for all $\b\geq \frac{\d}{2g}$,
 \be\label{normasuphx}
  \|h_x\|_\infty\leq 1+ ce^{-\b'}\,.
 \ee

Altogether, inequalities (\ref{primo})-(\ref{normasuphx})  imply that
\be\label{primo2}
\sum_{\t}\m_{i}(\t) p(\t)q(\t)\left(\m_{i+1}^{\t^+}
(\nabla_x g_{i+2})\right)^2\leq c_1
\m_{i}\left(\v_x (f)\right)\,,
\ee
with $c_1=c_1(\b, \D, g)= 1+O(e^{-c\b})$.
Thus, summing both sides of (\ref{1termine})
over $x\in L_i$ and $i\in\{0,\ldots,m\}$,
and applying  inequality (\ref{primo2}),   we obtain
\be\label{pvareresto}
\! \pv (f)\,\leq\, 2c_1\,\mathcal{D}(f)+
2\,\sum_{i=0}^{m}\sum_{x\in L_i}\m\!\left[
\sum_{\t}\m_{i}(\t)p(\t)q(\t) \left(\m_{i+1}^{\t^-}(h_x,
g_{i+2})\right)^2\!\right]\,.
\ee

Notice that since $g_{m+2}\equiv f$ is constant w.r.t. $\m_{m+1}$,
then $\m_{m+1}^{\t^-}(h_x, g_{m+2})\equiv 0$ and
the value $m$ can be removed from the summation over $i$
in the r.h.s. of (\ref{pvareresto}).\\
It now remains to  analyze  the covariance
$\m_{i+1}^{\t^-}(h_x, g_{i+2})$.

\subsubsection{Recursive argument}
Before going on with the proof, we
need some more definitions and notation. For every $x\in L_i$,
let $D_x$ denote the set of nearest neighbors
of $x$ in the level $L_{i+1}$ (descendants of $x$).
Given $x\in L_i$ and
$\ell\in\mathbb{N}$, let us define the following objects:
\begin{enumerate}
\item[(i)] $D_{x,\ell}:=\{y\in L_{i+1}\,:\,d(y,
D_x)\leq \ell \}\,$ is  the $\ell$-neighborhood of $D_x$ in
$L_{i+1}$;
\item[(ii)] $\mathcal{F}_{x,\ell}\,:=\,
 \s \left(\s_y\,:\, y\in
 B_{i+1}\setminus D_{x,\ell} \right)\,$
 is the $\s$-algebra generated by the spins on $B_{i+1}\setminus
 D_{x,\ell}$;
\item[(iii)]  $\m_{x,\ell}(\,\cdot\,):=\m\,(\cdot\,|\mathcal{F}_{x,\ell})\,$
is the Gibbs measure conditioned on the $\s$-algebra $\cF_{x,\ell}$.
\end{enumerate}
We remark that $D_{x,0}= D_x$, and that there exists some
$\ell_0\leq 2(i+1)$ such that, for all integers $ \ell\geq
\ell_0$, $D_{x,\ell}= L_{i+1}$ and  $\m_{x,\ell}=\m_{i+1}$.\\
We also remark that for any function $f\in
L^1(\Omega,\mathcal{F}_{i+1},\m)$, the set of variables
$\{\m_{x,\ell}(f)\}_{\ell\in \N}$ is a Martingale with respect to
the filtration $\{\mathcal{F}_{x,\ell}\}_{\ell=0,1,\ldots,\ell_{0}}$.
\vspace{6mm}

 Let us now come back to our proof and  recall the following
property of the covariance. For all subsets $D\subseteq C \subseteq B$,
 \be\label{pro2}
 \m_C^{\n}(f,g)\,=\,\m_C^{\n}(\m_D(f,g))+\m_C ^{\n}(\m_D (f),\m_D(g))\,.
\ee
Since the support of $\m_{i+1}$ strictly contains the support of
$\m_{x,0}$, we can apply the property (\ref{pro2}) to the square
covariance $(\m_{i+1}^{\t^-}(h_x, g_{i+2}))^2$ appearing in
(\ref{pvareresto}), in order to get
\be\label{decompcovar}
(\m_{i+1}^{\t^-}(h_x,
g_{i+2}))^2\,\leq\,2(\m_{i+1}^{\t^-}(\m_{x,0}(h_x,g_{i+2})))^2 +2
(\m_{i+1}^{\t^-}(\m_{x,0}(h_x), \m_{x,0}(g_{i+2})))^2 \,.
\ee
 The first term in the r.h.s. of (\ref{decompcovar}) can be bounded,
 by the Schwartz inequality, as
 \be\label{1terminesecondo}
(\m_{i+1}^{\t^-}(\m_{x,0}(h_x, g_{i+2})))^2 \leq  \m_{i+1}^{\t^-}
(\v_{x,0}(h_x))\cdot \m_{i+1}^{\t^-} (\v_{x,0}(g_{i+2}))\,.
\ee
 The second term can be rearranged and bounded as follows:
 \bea\label{2bisterminesecondo}
[\m_{i+1}^{\t^-}(\m_{x,0}(h_x),\m_{x,0}( g_{i+2}))]^2
&=&
\left[\m_{i+1}^{\t^-}\left(\,\m_{x,0}(h_x)-\m_{i+1}(h_x),
g_{i+2}\,\right)\right]^2\hspace{1cm}\nonumber\\
&=&
\left[\m_{i+1} ^{\t^-}\left(\sum_{\ell=1}^{\ell_0}\left(\m_{x,\ell-1}(h_x)
-\m_{x,\ell}( h_x)\right) ,g_{i+2}\right)\right]^2 \nonumber\\
&\leq&
\sum_{\ell=1}^{\ell_0} \ell^2\,\left[\m_{i+1}
^{\t^-}(\m_{x,\ell-1}(h_x)
-\m_{x,\ell}(h_x) ,g_{i+2})\right]^2 \nonumber\\
&=&
\sum_{\ell=1}^{\ell_0}\ell^2\, \left[\m_{i+1}
^{\t^-}\left(\m_{x,\ell}(\m_{x,\ell-1}(h_x)
,g_{i+2})\right)\right]^2 \,,\hspace{1.3cm}
\eea
where in the second
line, due to the fact that $\m_{x,\ell_0}=\m_{i+1}$ for some
$\ell_0$, we substituted $\m_{x,0}(h_x)-\m_{i+1}(h_x)$ by the
telescopic sum $\sum_{\ell=1}^{\ell_0}(\m_{x,\ell-1}(h_x)
-\m_{x,\ell}(h_x))\,$.
Applying again the Cauchy-Schwartz inequality
to the last term in (\ref{2bisterminesecondo}), we  get
$$
[\m_{i+1}^{\t^-}(\m_{x,0}(h_x),\m_{x,0}( g_{i+2}))]^2\leq\hspace{7cm}
$$\vspace{-7mm}
\be\label{2terminesecondo}
\hspace{1.9cm}\leq\,
\sum_{\ell=1}^{\ell_0}\ell^2\, \m_{i+1}
^{\t^-}\left(\v_{x,\ell}(\m_{x,\ell-1}(h_x))\right)\cdot
\m_{i+1}^{\t^-}\left(\v_{x,\ell}(g_{i+2})\right)
\ee
 To conclude
the estimate on the covariance, it remains to analyze the  three
quantities appearing in (\ref{1terminesecondo}) and
(\ref{2terminesecondo}):
\begin{enumerate}
\item[(i)] $\m_{i+1}^{\t^-}\left(\v_{x,\ell}(g_{i+2})\right)$,
for all $\ell=0,1,\ldots,\ell_0$;
\item[(ii)] $\m_{i+1}^{\t^-} (\v_{x,0}(h_x))$;
\item[(iii)] $\m_{i+1}
^{\t^-}\left(\v_{x,\ell}(\m_{x,\ell-1}(h_x))\right)$, for all
$\ell=1,\ldots,\ell_0$\,,
\end{enumerate}
We proceed estimating separately these three terms.\vspace{5mm}

 \no {\bf First term: Poincar\'{e} inequality for the
marginal measure on $D_{x,\ell}$}.\\
\no Let us consider the variance $\v_{x,\ell}(g_{i+2})$ appearing in (i).
By definition, the function $g_{i+2}$  depends on the spin
configuration on $B_{i+1}$. Since the measure $\m_{x,\ell}^\n$
fixes the configuration  on $B_{i+1}\setminus D_{x,\ell}$, it follows that
$$
\m_{x,\ell}^\n(g_{i+2})=\m_{{x,\ell}|_{D_{x,\ell}}}^\n(g_{i+2})\,.
$$
Thus, for every configuration $\n\in \Omega^{+}$, we can apply the
Poincar\'e inequality stated in Theorem (\ref{teoPoin}) to
$\v_{x,\ell}^\n(g_{i+2})$, and  obtain the inequality
\be\label{primadisug}
\m_{i+1}^{\t^-}\left(\v_{x,\ell}(g_{i+2})\right)\,\leq\,c_0
\sum_{y\in D_{x,\ell}}\m_{i+1}^{\t^-}(\v_{K_y}(g_{i+2})),
\ee
with $c_0=1+O(e^{-c\b})$ independent of the size of system.\vspace{5mm}

\no {\bf Second term: computation of the variance of $h_x$}.\\
\no Notice that, from definition (\ref{def:density}),
it turns out that $h_x$  only depends on the spin configuration on $D_x$.
 In particular, for all $\n$ which agrees with $\t^-$ on $B_{i}$,
 $\m_{x,0}^\n(h_x)=0$ and
\be\label{varianzadensita}
\v_{x,0}^\n(h_x)\leq
\|h_x\|^2_\infty - 1\,.
 \ee
Together with inequality (\ref{normasuphx}), this yields
the bound:
 \be\label{varianza2h_x}
 \v_{x,0}^\n(h_x)\leq ce^{-\b'}=:k_{\b} \,.
 \ee
 \vspace{2mm}

\no {\bf Third term: the variance of $\m_{x,\ell-1}(h_x)$}.\\
\no We now consider the variance
$\v^\n_{x,\ell}(\m_{x,\ell-1}(h_x))$, with $\n\in \Omega^+$ and
$\ell\geq 1$. Applying the result of Theorem \ref{teoPoin}, we obtain
 \bea\label{ultima}
\v_{x,\ell}^\n(\m_{x,\ell-1}(h_x))\,
&\leq&
\, c_0\sum_{z\in D_{x,\ell}}
\m_{x,\ell}^\n( \v_{K_z} (\m_{x,\ell-1}(h_x)))\nonumber\\
&=&
\, c_0\sum_{z\in D_{x,\ell}\setminus D_{x,\ell-1}} \m_{x,\ell}^\n(
\v_{K_z} (\m_{x,\ell-1}(h_x)))\,,
\eea
where in the last line we used
that the function $\m_{x,\ell-1}(h_x)$ does not depend
on the spin configuration on $D_{x,\ell-1}$.

 Let $z\in
D_{x,\ell}\setminus D_{x,\ell-1}$, and for any configuration
$\zeta\in\Omega_{D_{x,\ell}}^\n$, let us denote by $\zeta^+$ and
$\zeta^-$ the configurations that agree with $\zeta$ in all sites
but $z$, and have respectively a $(+)$-spin and a $(-)$-spin on $z$.
The summand in (\ref{ultima}) can be trivially bounded as
\be
\m_{x,\ell}^\n( \v_{K_z} (\m_{x,\ell-1}(h_x)))\,\leq\,
\frac{1}{2}\sup_{\,\,\,\,\zeta\in\,\Omega_{x,\!\ell}^\n}
(\m_{x,\ell-1}^{\zeta^+}(h_x)-\m_{x,\ell-1}^{\zeta^-}(h_x))^2\,.
 \ee
Moreover, by stochastic domination, the inequality
$\m_{x,\ell-1}^{\zeta^+}(h_x)\geq\m_{x,\ell-1}^{\zeta^-}(h_x)$ holds.
Let $\nu(\s,\s')$ denote a monotone coupling  with marginal
measures $\m_{x,\ell-1}^{\zeta^+}$ and $\m_{x,\ell-1}^{\zeta^-}$. We
then have
\begin{eqnarray}\label{correl1}
\m_{x,\ell-1}^{\zeta^+}
(h_x)-\m_{x,\ell-1}^{\zeta^-}(h_x)&=&\sum_{\s,\s'}
\nu(\s,\s')\left(h_x(\s)- h_x(\s')\right)\nonumber\\
&\leq &  \|h_x\|_\infty \,\nu(\s_y\ne\s'_y\,,
y\in D_{x} )\nonumber\\
&\leq& \,\D\|h_x\|_{\infty} \,\max_{y\in D_x}\left(
\nu(\s_{y}=+)-\nu(\s'_{y}
=+)\right)\nonumber\\
&=& \,\D\|h_x\|_{\infty}\,\max_{y\in D_x}(\m_{x,\ell-1}^{\zeta^+}
(\s_{y}=+)-\m_{x,\ell-1}^{\zeta^-}(\s_{y}=+)),\quad\quad\quad
\end{eqnarray}%
where we used that  the function $h_x$ only depends on the
spins on $D_{x}$.

From Proposition \ref{propdecadimento}, with  $U= F_{i+2}\cup D_{x,\ell}$,
and since $d(z,y)\geq d(z,D_x)\geq\ell$,
the probability of disagreement appearing in (\ref{correl1})
can be bounded as
 \bea\label{iterato}
\m_{x,\ell-1}^{\zeta^+}(\s_{y}=+)-\m_{x,\ell-1}^{\zeta^-}(\s_{y}=+)
\leq
c\,e^{-\b'\ell}\,. \eea

Putting together formulas
(\ref{ultima})-(\ref{iterato}), and applying inequality
(\ref{normasuphx}),
  we  obtain that for all $\n\in\Omega_+$,
\be\label{3termine}
\v^\n_{x,\ell}(\m_{x,\ell-1}(h_x)) \,\leq\, k'_{\b}e^{-2\b'\,\ell}\,
\ee
with $k'_{\b}=c(1+O(e^{-c\b})) \,.$ \vspace{6mm}

\no {\bf Conclusion}.\\
Let us go back to inequalities (\ref{1terminesecondo}) and
(\ref{2terminesecondo}). Applying bounds
(\ref{primadisug}),(\ref{varianza2h_x}) and (\ref{3termine}), we get
respectively
\begin{itemize}
\item $\displaystyle(\m_{i+1}^{\t^-}(\m_{x,0}(h_x, g_{i+2})))^2 \leq k_\b
\sum_{y\in D_x} \m_{i+1}^{\t^-}(\v_{K_y} (g_{i+2}))$,\vspace{-4pt}
\item $\displaystyle[\m_{i+1}^{\t^-}(\m_{x,0}(h_x),\m_{x,0}( g_{i+2}))]^2
\leq k'_{\b}\sum_{\ell=1}^{\ell_0}\ell^2 e^{-2\b'\,\ell} \sum_{y\in
D_{x,\ell}} \m_{i+1}^{\t^-}(\v_{K_y} (g_{i+2}))$,
\end{itemize}
where we included in $k_{\b}$ and $k'_{\b}$ all constants non
depending on $\b$.

 For all $\b\gg 1$, there exists a constant $\varepsilon\equiv\varepsilon(\b,g, \Delta)=
 O(e^{-c\b})$ such that
$k_\b\leq\varepsilon$ and $ k'_\b \ell^2 e^{-\b' \ell}\leq k'_\b
e^{-\b'} \leq\varepsilon$. Substituting $\varepsilon$ in the
inequalities above and  summing the two terms  as in
(\ref{decompcovar}), we obtain
$$
\left(\m_{i+1}^{\t^-}(h_x, g_{i+2})\right)^2\leq \varepsilon
\sum_{\ell=0}^{\ell_0} e^{-\b'\ell} \sum_{y\in D_{x,\ell}}
\m_{i+1}^{\t^-}(\v_{K_y} (g_{i+2}))\,.$$ Inserting this result in the
second term of formula (\ref{pvareresto}) and  rearranging the
summation, we get
$$\sum_{i=0}^{m-1}\sum_{x \in
L_i}\m\left[\sum_{\t}\m_i(\t)p(\t)q(\t)\left(\m_{i+1}^{\t^-}(h_x,
g_{i+2})\right)^2\right]\hspace{4cm}$$\vspace{-5mm}
\bea\label{meta} &\leq& \varepsilon \sum_{i=0}^{m-1}\sum_{x \in
L_i}\sum_{\ell=0}^{\ell_0}\sum_{y\in D_{x,\ell}}e^{-\b'\ell} \m(\v_{K_y}
(g_{i+2}))\nonumber\\
&\leq& \varepsilon \sum_{i=0}^{m-1}\sum_{y\in L_{i+1}} \m(\v_{K_y}
(g_{i+2}))\sum_{\ell=0}^{\ell_0} e^{-\b'\ell}n(\ell) \,,
 \eea
where in the last line we denoted by $n(\ell)$ the factor which
bounds the number of vertices $x$ such that a fixed vertex $y$
belongs to $ D_{x,\ell}$. Since $n(\ell)$ grows at most like $\Delta^\ell$, the
product $e^{-\b' \ell} n(\ell)$ decays exponentially with $\ell$ for
all $\b\gg 1$. Thus the sum over $\ell\in \{0,\ldots,\ell_0\}$ can
be bounded by a finite constant $c$ which will be included  in the
factor $\varepsilon$ in front of the summations. Continuing from
(\ref{meta}),  we get \bea\sum_{i=0}^{m-1}\sum_{x \in
L_i}\m\left[\sum_{\t}\m_i(\t)p(\t)q(\t)\left(\m_{i+1}^{\t^-}(h_x,
g_{i+2})\right)^2\right]&\leq& \varepsilon
\sum_{i=1}^{m}\sum_{y\in L_{i}}\m(\v_{K_y} (g_{i+1}))\nonumber\\
& \leq& \varepsilon\, \pv(f)\,.\eea
 Inserting this result in (\ref{pvareresto}) and noticing that
$\varepsilon=O(e^{-c\b})<1$ for $\b$ large enough, we obtain
$$\pv(f)\leq 2c_1\cD(f) + \varepsilon \pv (f)\Longrightarrow
\pv(f)\leq \frac{2c_1}{1-\varepsilon}\,\cD(f)\,.$$
Together with inequality
(\ref{pvar}), this implies that
$$\v(f)\leq c_0\pv(f)\leq c_2\cD(f) \,,$$ that is the
desired Poincar\'{e} inequality with
$c_2=c_2(\b,g,\Delta) =2(1+O(e^{-c\b}))$
independent of the size of the system .
Notice also that the lower bound on the spectral gap, $1/c_2$,
increases with $\b$ and converges to $1/2$ when $\b\uparrow\infty$.
This concludes the proof of Theorem \ref{teo}.
{\hfill$\square$\vskip.5cm}

\section{Influence of boundary conditions on the mixing time}
In this section we discuss two examples of  influence
of the boundary condition on the mixing time
derived from Theorem \ref{teo}.
In particular, we first prove Lemma \ref{lemma:hyper},
which implies the applicability of Theorem \ref{teo}
to hyperbolic graphs with sufficiently high degree.
Then we prove Theorem \ref{th:esempio} providing an explicit
example of growing graph which exhibits the behavior stated in the
theorem.

\subsection{Hyperbolic graphs}
In this subsection we want to prove Lemma \ref{lemma:hyper}.
Before starting the proof, let us recall
the following definition:
\begin{definition}\label{def:numberend}
 The number of ends, $\mathcal E (G)$, of a graph
 $G=(V,E)$, is defined as
$$\mathcal{E}(G):=\sup_{ K\subset V\atop K
{\scriptsize\mbox{finite}}}\{ \mbox{ number of infinite connected
components of } G\setminus K\,\}\,,$$ where $G\setminus K$ denotes
the graph obtained from  $G$ by removing the vertices which belong
to $K$ and the edges incident to these vertices.
\end{definition}
It is well known that hyperbolic graphs have \emph{one-end},
as well as all the lattices $\Z^d$ with $d\geq 2$.
This property, together with the planarity and the cycle-periodic
structure of the hyperbolic graphs, will be the main ingredient
of the next proof.

\begin{proof}[Proof of Lemma \ref{lemma:hyper}]
Consider a planar embedding of $\H(v,s)$  and assume,
without loss of generality, that the vertices of any
ball $B_r$  around a fixed vertex $o\in V$,
are in the infinite face of the graph.

With the same notation introduced in the previous sections,
let us consider  $x\in L_{r}= \partial_V B_{r-1}$.
Let $D_x= N_x \cap L_{r+1}$ be the set of "descendants of $x$",
and define
$$S_x:=N_x\cap L_{r}\quad \mbox{ and }\quad P_x:= N_x\cap L_{r-1}\,,$$
so that $N_x = D_x \cup S_x \cup P_x$.
Thus,  when specialized to hyperbolic graphs $\H(v,s)$,
the definition \ref{def:growing} of growing graph
corresponds to the condition:
\be\label{growing2}
\min_{x\in L_r,r\in\N}\{D_x -S_x- P_x\}=
\min_{x\in L_r,r\in\N}\{ v -2(S_x+P_x)\}=g\,.
\ee

Let us start verifying the  following general properties:
\begin{itemize}
\item[(i)] Every vertex $x\in L_r$, has at least one neighbor in $L_{r+1}$, i.e. $|D_x|\geq 1$;
\item[(ii)] Every vertex $x\in L_r$, has at most two neighbor in $L_{r}$, i.e. $|S_x|\leq 2$;
\item[(iii)] Every vertex $x\in L_r$, has at most two neighbors in $L_{r-1}$, i.e. $|P_x|\leq 2$.
\end{itemize}

Property (i) is proved by contradiction.
Assume the existence of a vertex $x\in L_r$ such that
$D_x=\emptyset$. By the periodic structure of the graph
it follows that there are at least $v$ vertices in $L_r$
with no descendants, spaced out by vertices in $L_r$ connected
to $\H(v,s)\setminus B_r$  by at least one edge.
Then, it exists a finite path, from $x$ to another vertex in $L_r$
with no descendants, dividing  two infinite components.
This is in contradiction we the property of having one end,  thus
we conclude that $|D_x|\geq 1$.

Property (ii) is  also proved contradiction.
Assume the existence of a vertex $x\in L_r$ linked
to three vertices in $L_r$. Then there exist two faces, included in $B_r$,
both passing through $x$ and one of its neighbors on $L_r$, say $y$
(see Fig. \ref{fig:intornox}, left frame).
Consequently, or $D_y\ne\emptyset$, contradicting the planarity of the graph
or the property of having one-end, or $D_y=\emptyset$,
contradicting property (i). We conclude that $|S_x|\leq 2$.

Property (iii) is clearly satisfied if $v=3$, as a consequence of property (i).
Let $v\geq 4$ and proceed by induction. For all vertices $x$ in the first level,
we have obviously that  $|P_x|=1$. Assume that property (iii) holds for
all vertices of $L_{r-1}$, and consider a vertex $x\in L_r$.
If we assume by contradiction that $P_x=3$,
then there exist two faces, included in $B_r$, both passing through $x$
and one of its parents, say $y$ (see Fig. \ref{fig:intornox}, right frame).
Consequently, the vertex $y$ can not have neighbors in $L_r$ other than $x$,
because this would contradict the planarity of the graph
or the property of having one-end.
On the other hand, if $s>3$,  $y$ can not be linked to the other parents of $x$
because  this would create a triangular face.
Globally, and by the inductive hypothesis that $|P_y|\leq 2$,
it  follows that $|N_y|\leq3$, which is a contradiction.
Then $P_x\leq 2$, which completes the induction step for $s>3$.
If $s=3$,  $y$ is connected with the two other parents
of $x$ to create two triangular faces.
Thus, by the inductive  hypothesis that $|P_y|\leq 2$,
it  follows that $|N_y|=5$, which is a contradiction since
$v\geq 7$  for all triangular tilings (recall the condition $(v-2)(s-2)>4$).
Then $|P_x|\leq2$, and by induction this concludes the proof of property (iii) also for $s=3$.
\begin{figure}[htb]
\begin{center}
\psfrag{a}{$o$} \psfrag{b}{$x$} \psfrag{c}{$y$}
\psfrag{e}{$L_{r}$} \psfrag{f}{$L_{r-1}$}
\psfrag{g}{$o$}
\psfrag{h}{$y$} \psfrag{i}{$x$} \psfrag{l}{$L_{r}$}
\includegraphics[width=5.5cm]{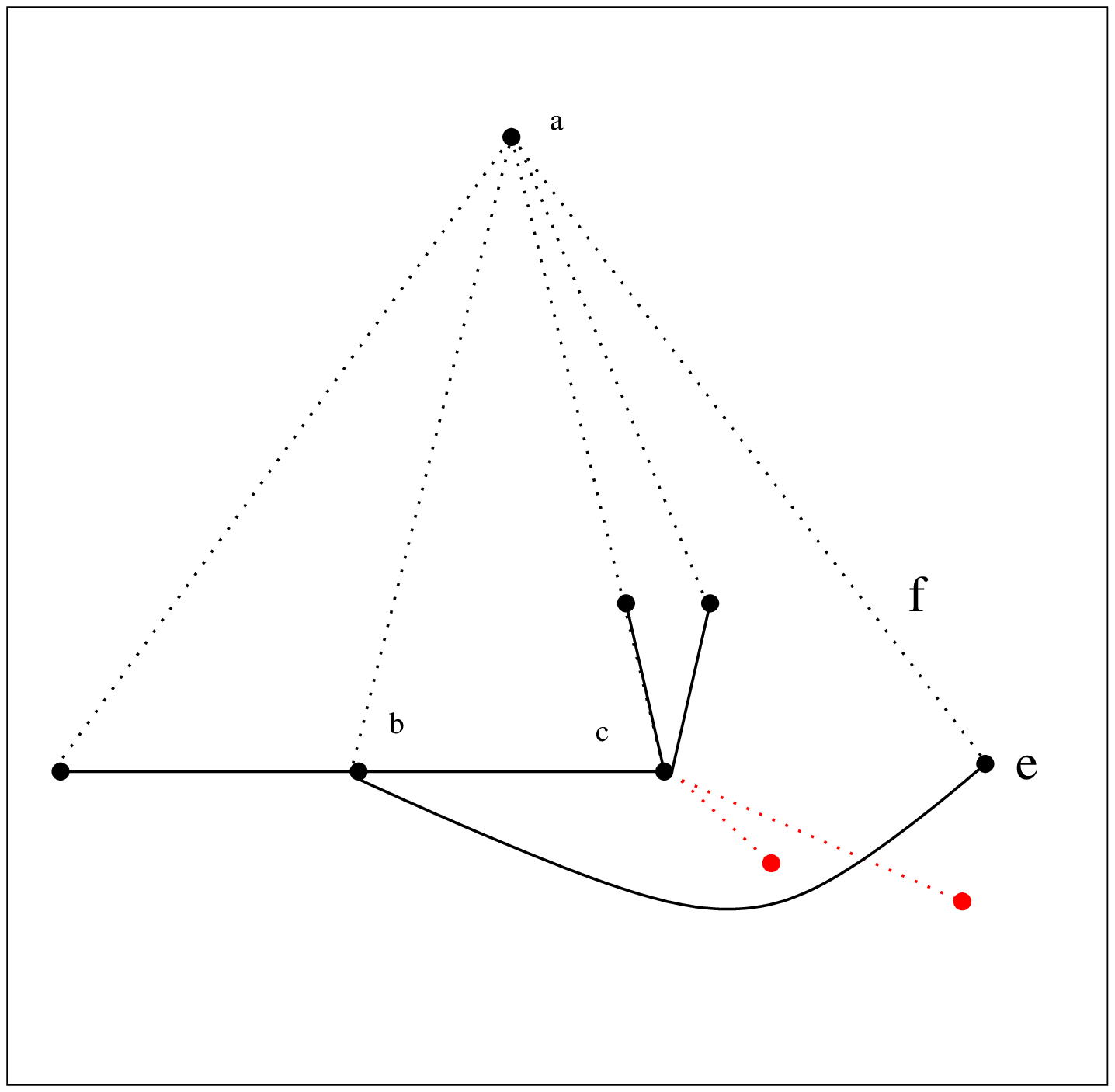}
\hspace{1cm}
\includegraphics[width=5.09cm]{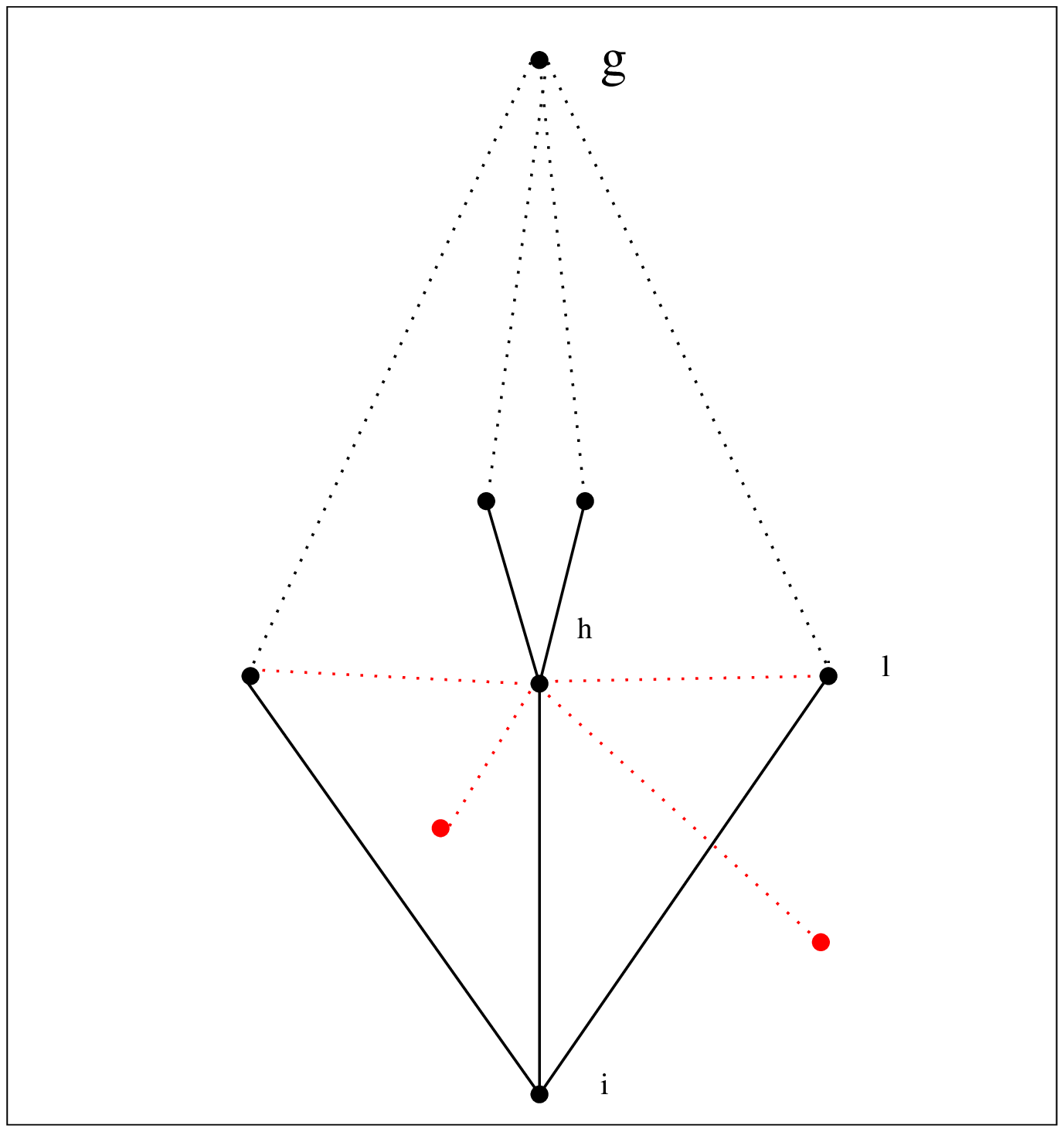}
\end{center}
\caption{{\small The presence of more than two neighbors of $x$
in the same level of $x$ (left frame)
or in the level above $x$ (right frame) is not
consistent with the structure of hyperbolic graphs.
The red lines correspond to  edges which are forbidden
by the geometric properties of the graph.}}\label{fig:intornox}
\end{figure}

Properties (ii) and (iii) imply that $S_x+P_x\leq4$.
From (\ref{growing2}), we then get that for all $v\geq 9$,
$\H(v,s)$ is a growing graph with parameter $g\geq v-8$.
Moreover, it can be easily verified that if $s=3$ (triangular tilings), then  $g= v-8$,
since in this case $S_x+P_x=4$ for some $x\in V$.

If $s>4$, the above condition can be improved by proving the following properties:
\begin{itemize}
\item[(iv)] If $s= 2k$, with $k\geq 2$, then for every $x\in L_r$, $S_x=0$;
\item[(v)] If $s= 2k+1$, with $k\geq 2$, then for every $x\in L_r$, $S_x+P_x\leq 2$.
\end{itemize}

Let us prove property (iv) by induction.
For all vertices $x$ in the first level, we have obviously that $|S_x|=0$,
otherwise we would have a triangular face. Assume that property (iv)
holds for all vertices of $L_{r-1}$, and consider a vertex $x\in L_r$.
If we assume that $x$ has a neighbor in $L_r$, say $y$,
then there exists a face $F$, included in $B_r$, passing through
$x$, $y$, and the respective parents (see Fig. \ref{fig:evenodd}, left frame).
The remaining sides of $F$ are all included in $B_{r-1}$,
and by the inductive hypothesis they can only connect vertices on
subsequent levels. In particular, this face can only have an odd number
of sides, which is in contradiction with the hypothesis that $s=2k$.
We conclude that $S_x=0$  for all $x\in V$.

Property (v) is clearly satisfied if $v=3$, as a consequence of property (i).
Thus let $v\geq 4$ and assume by contradiction that $|S_x|+|P_x|=3$, which means,
coherently to properties (ii) and (iii), that $S_x=2$ and $P_x=1$ or viceversa.
In both cases, there exist two faces, included in $B_r$,
passing through $x$ and one of its parents,
say $y$ (see Fig. \ref{fig:evenodd}, central and right frames).
Then the vertex $y$ can not have neighbors in $L_r$ other than $x$,
because this would create a triangular face, or would contradict
the planarity of the graph or the property of having one-end.
Analogously, $y$ can not have neighbors in $L_{r-1}$, because
this would create a triangular or a square face,
while $s=2k+1$ by hypothesis.
Then we get that $|N_y|=1+|P_y|\leq 3$, which is in contradiction with
the hypothesis that $v\geq 4$. We conclude that $|S_x|+|P_x|\leq 2$.

\begin{figure}[htb]
\begin{center}
\psfrag{a}{$x$} \psfrag{b}{$y$} \psfrag{c}{$L_r$}
\psfrag{d}{$L_{r-1}$} \psfrag{f}{$F$}
\psfrag{e}{$o$}
\includegraphics[width=4.5cm]{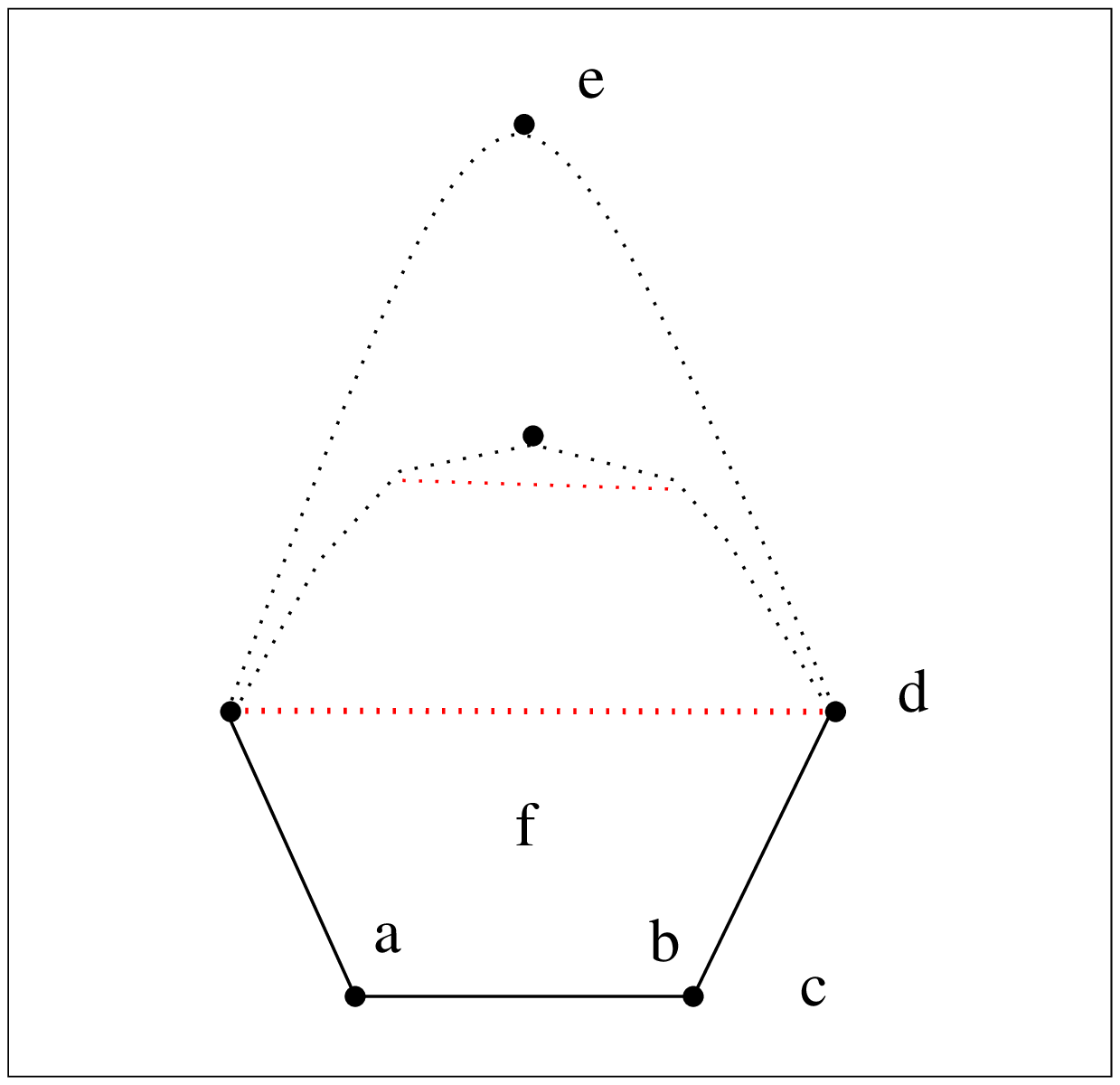}
\hspace{0.3cm}
\includegraphics[width=4.55cm]{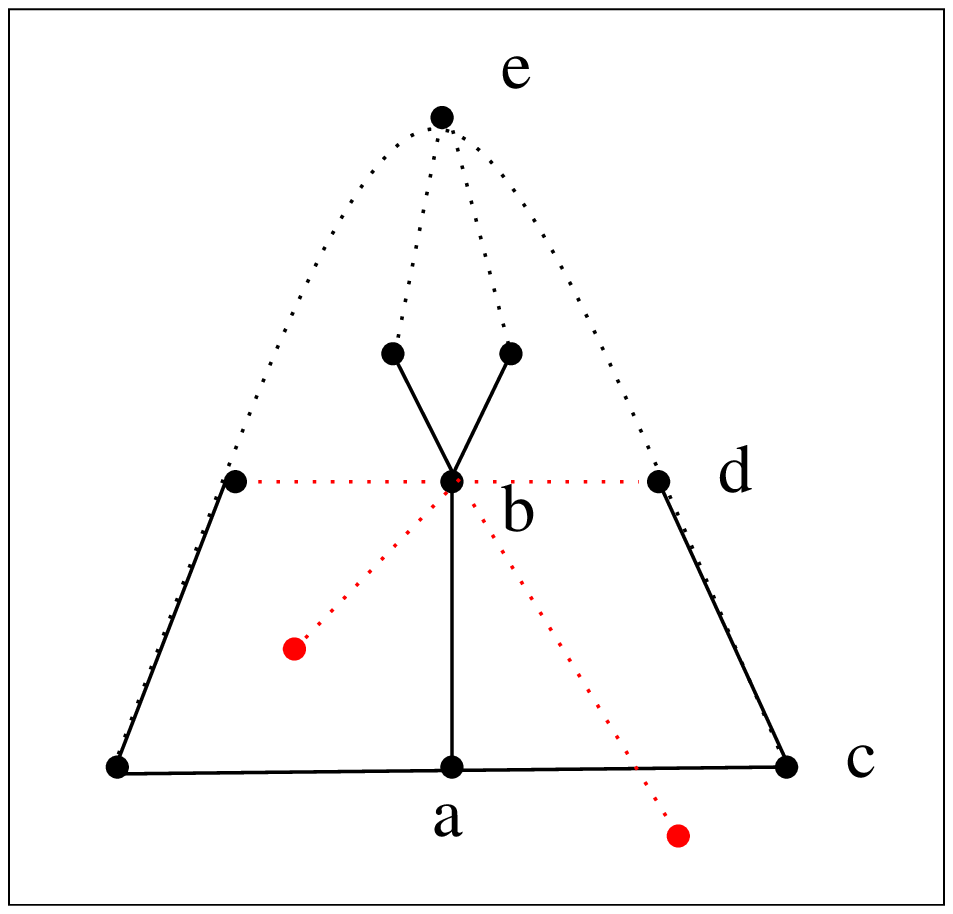}
\hspace{0.3cm}
\includegraphics[width=4.5cm]{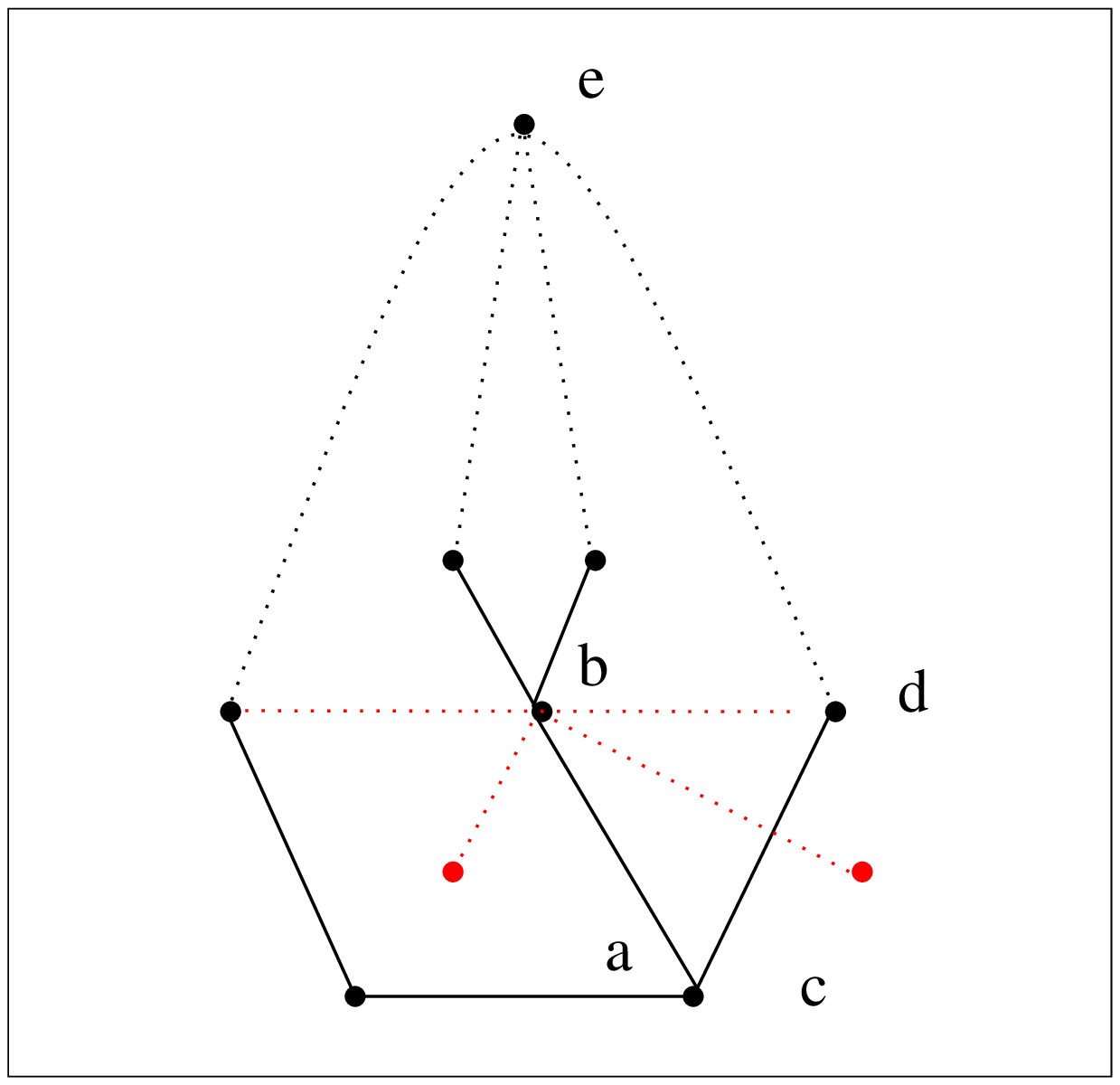}
\end{center}
\caption{{\small Forbidden patterns.
In the left frame it is shown that $S_x\ne\emptyset$
only if the number  of side in each face is odd.
The right and central frames show that the condition
$|S_x|+|P_x|\geq 3$ is possible only if $s=3$.
}}\label{fig:evenodd}
\end{figure}
All together,  properties (iii)-(v)  imply that if $s\geq 4$,
then $S_x +P_x\leq 2$. Thus, for all couples
$(v,s)$ such that $s\geq 4$ and $v\geq5$,
$\H(v,s)$ is a growing graph with parameter $g\leq v-4$.
It can be easily verified that if $s\geq 4$ then $g=v-4$, since $S_x+P_x=2$
for some $x\in V$.
\end{proof}

Together with Theorem \ref{teo}, the above lemma  implies that for all
$(v,s)$ such that $\H(v,s)$ is a growing graph and for all $\b\gg 1$,
 the dynamics on an $n$-vertex ball of $\H(v,s)$  with $(+)$-boundary condition,
has spectral gap uniformly positive in $n$ and mixing time at most linear in $n$.
Comparing this result with the behavior of the dynamics with free boundary condition
\cite{BKMP}, we get a convincible example of the influence
of boundary conditions on the mixing time.

\subsection{Expanders}
In this subsection we prove Theorem \ref{th:esempio}
by providing an explicit example of growing graph
which exhibits the behavior stated in the theorem.

\subsubsection{Construction}
Let us start with some definitions.

Let $G=(V,E)$ be a finite graph. The edge isoperimetric
constant of $G$, defined in (\ref{isoperimetrica2}) for infinite graphs,
is given by
\be\label{def:cheegerfiniteG}
i_e(G):= \min_{\emptyset\ne S\subset V\atop |S|\leq \frac{|V|}{2}}
\left\{\frac{|\partial_E (S)|}{|S|}\right \}\,.
\ee
For $c>0$ and $k,n\in\N$, we then have the following definition:
\begin{definition}
A finite graph $G=(V,E)$ is called an $(n,k,c)$-expander if it is regular with degree $k$,
$|V|= n$,  and $i_e(G)=c$.
\end{definition}
Notice that, since $G$ is finite,  $i_e(G)>0$ if and  only if $G$ is  connected.
In particular, every connected $k$-regular graph on $n$ vertices
is an expander for some $c>0$.
However, one is usually interested in a family of expanders, that is
a sequence of $(n,k,c)$-expander graphs such that $k$ is fixed,
$c\geq \e>0$, and $n\rightarrow\infty$.  That $c\geq\e>0$ ensures
that these graphs are highly connected, or, in other terms,  that
the sequence of expanders converges to an infinite regular nonamenable graph.
It is easy to see that  this  happens only if $k\geq 3$.
On the other hand, for all $k\geq 3$,  a sequence of $(n, k, c)$-expanders
exists, as it has been proved by showing that  $k$-regular random graphs on $n$ vertices
are expanders with probability tending to $1$
as $n\rightarrow\infty$ (probabilistic method).
See \cite{HLW} for a nice survey on expanders and their applications.

Let $T^\D$ denote an infinite rooted tree with constant degree $\D$
and root $o\in V$, and assume that $\D\geq 6$.
Fix an integer $d <\D-2$, and  connect the  vertices
on each level   $L_r=\{x\in V: d(x,o)=r\}$ of the tree
 in such a way that induced subgraph on $L_r$ is an
$(\D(\D-1)^{r-1}, k_r, c_r)$-expander,
with $3\leq k_r\leq d$ and isoperimetric constant $c_r$
uniformly positive in $r$ .
We denote the infinite graph  obtained with this procedure
 by $X^{\D,d}$ (see Fig. \ref{fig:expander}).

\begin{figure}[htb]
\begin{center}
\includegraphics[width=6cm]{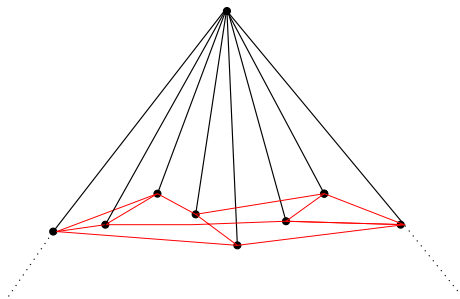}
\end{center}
\caption{{\small  Construction of $X^{8,3}$ to first level.
The red lines correspond to the edges of the expander in $L_1$.}}\label{fig:expander}
\end{figure}

 Notice that $X^{\D,d}$ has degree $D\geq\Delta+d$, and
 by construction it is $(g,o)$-growing graph with $g\geq\D-2-d$.
Moreover,  one can easily realize that the isoperimetric constant
$i_e(B_r)$ of the ball $B_r \subset X^{\D,d}$ centered in $o$,
is uniformly positive  in $r\geq 1$, namely, there is a constant
$\e=\e(\D,d)>0$ such that $i_e(B_r)>\e$  for all $r\in\N$.

\subsubsection{Free boundary dynamics finite balls of $X^{\D,d}$}
Let us consider the Glauber dynamics on the $n$-vertex
ball $B\equiv B_m$ of the graph $X^{\D,d}$ with free boundary condition.
We use the same notation of section \ref{sec:glauber}, but
here $\m$  denotes the Gibbs measure over $\Omega\equiv\Omega_B=\{\pm 1\}^B$
with free boundary condition, that is, for all $\s\in\Omega$,
\be\label{gibbsfreeboundary}
\m(\s)=\frac{1}{Z(\b)}\exp (\b \sum_{(x,y)\in E(B)}\s_x\s_y)\,.
\ee
With this notation, we can state the following result.
\begin{proposition}\label{prop:mixingexp}
Let  $B$ the  $n$-vertex ball of the graph $X^{\D,d}$.
Then, for all $\b\gg 1$, the Glauber dynamics on $B$ with free boundary condition
and zero external field, has spectral gap $O(e^{-\th n})$,
with $\th= \th(\D,d,\b)>0$.
\end{proposition}

\begin{proof}
For every $\s\in\Omega$,
let  $m_B(\s):=\sum_{x\in B}\s_x$ be the
magnetization of a configuration $\s$, and define the following
characteristic function:
\be\label{testfunction}
\1_{\{m_B >0\}}(\s)\,=\,\left\{\begin{array}{cc}1 & \mbox{ if } m_B(\s)>0\\
0 & \mbox{ if } m_B(\s)\leq 0 \end{array}\right.\,.
\ee
Without loss of generality we can restrict the analysis to odd  values of $n$.
Since in this case $m_B(\s)\ne 0$ for all $\s$, by symmetry it follows that
$\v (\1_{\{m_B >0\}}) = 1/4$ and then, from (\ref{cgap}),
\be\label{boundgap}
\cg(\m)\leq 4 \mathcal D(\1_{\{m_B>0\}})=
2\sum_{x\in B}\m\left(c_x \,[\nabla_x (\1_{\{m_B
>0\}})]^2 \right) \,.
\ee
Notice that $\nabla_x (\1_{\{m_B >0\}})(\s)\ne 0$ only if
$\s$ is such that $|m_B(\s)|= 1$, namely only if there are $(n+2)/2$ spins
in $\s$ with the same value, and
only if $x$ is one of the  $(n+2)/2$ vertices with spin of the same
sign of $m_B(\s)$.
Then we get
\be\label{bo2}
\mathcal{D}(\1_{\{m_B
>0\}})\,\leq\, \sfrac{n+2}{4}\,\,\m( |m_B|= 1)\,=\,
\sfrac{n+2}{2}\,\,\m( m_B= 1)\,.\ee
For every configuration $\s\in\Omega$, we define the subsets
$$
A^+ (\s)\,:=\,\{x\in B \,\mbox{ s.t.
} \s_x=+\}\quad\mbox{ and }\quad A^- (\s)\,=\,\{x\in B\,\mbox{ s.t.
} \s_x=-\}\,.
$$
Notice that the condition $m_B(\s)=1$ is satisfied if and only if
$|A^+(\s)|= (n+1)/2$ and  $|A^-(\s)|\,=\,(n-1)/2$.
Thus, any configuration $\s$ with  magnetization equal to $1$, is
univocally correspondent to a partition of the vertex set of $B$ into two
subsets $S$ and $T$ of size $(n+1)/2$ and $(n-1)/2$, respectively,
such that $A^+(\s)=S$ and $A^-(\s)=T$.
Let $\mathcal P$ denote the set of these partitions, and define $E(S,T)$
as the set of edges between $S$ and $T$.
Then we have
\bea\label{zeromagn}
\m( m_B =1)&=&
\sum _{\s\,: m_B(\s)=1} \frac{\exp(-\b \sum_{(x,y)\in E(B)}\s_x\s_y)}{Z(\b)}
\nonumber\\&=&
\sum_{(S,T)\in\mathcal P} \frac{\exp(\,\b(|E(S)|+|E(T)|-|E(S,T)|)\,)}{Z(\b)}
\nonumber\\&\leq &
\sum_{(S,T)\in\mathcal P}
\frac{\exp(\,\b(|E(S)|+|E(T)|-|E(S,T)|)\,
)}{\exp(\,\b(|E(S)|+|E(T)|+|E(S,T)|)\,)}\nonumber \\
&=& \sum_{(S,T)\in\mathcal P} \exp(-2\b |E(S,T)|)\,.
\eea
Notice that $|E(S,T)|=|\partial_E T|$.
In particular,  since $B$ has uniformly positive isoperimetric constant,
namely  $i_e(B)\geq \e>0$ for all $n\in\N$,
$$|E(S,T)|\geq \e|T|=\sfrac{\e}{2}(n-1)\,.$$
Continuing from (\ref{zeromagn}), we get
\bea
\m( m_B =1)&\leq&
\sum_{(S,T)\in\mathcal P} e^{-\b \e(n-1)}\nonumber\\
&=&\binom{n}{\frac{n-1}{2}}e^{-\b \e(n-1)}\nonumber\\
&\leq & 2^n e^{-\b \e(n-1)} \nonumber\\
&=& ce^{-(\b\e-\log 2)n}\,,
 \eea
where in the third line we approximated the binomial factor using
the Stirling formula.
Inserting this bound in (\ref{bo2}) we get that, for all
$\b>\frac{\log 2}{\e}$, there exists a
positive  constant $\th=\th(\D,d,\b) $ such that
\be
\mathcal{D}(\1_{\{m_B>0\}})\,\leq\,
\sfrac{n+2}{2}ce^{-(\b\e -\log 2)n}
\leq e^{-\th n},
\ee
which implies, by (\ref{boundgap}), that $ \cg(\m)= O(e^{- \th n} )$.
\end{proof}

Theorem \ref{th:esempio} follows from the above proposition,
since by construction the infinite  graph $X^{\D,d}$ is
growing with parameter $g\geq \D-d-2$.
Putting together Theorems \ref{teo} and Proposition \ref{prop:mixingexp},
we get the first rigorous example in which the mixing time
is at most linear in $n$ for the $(+)$-boundary condition,  while at
least exponential in $n$ for the free boundary condition.

\noindent {\bf Acknowledgements.}
I am grateful to an anonymous referee for very useful
comments and suggestions on earlier draft. I also would like to thank  the
Department of Mathematics of the University of Roma Tre which
offered financial support and a friendly environment during my PhD
studies. I am especially thankful to Fabio Martinelli who introduced
me to this subject and provided useful insights.

\end{document}